\documentclass[a4paper, 12pt]{amsart}
\usepackage[T1]{fontenc}
\usepackage{textcomp }
\usepackage{amsmath}
\usepackage{amsfonts}
\usepackage[utf8]{inputenc}
\usepackage{stmaryrd}
\usepackage{tikz-cd}
\usepackage{amsmath, amssymb}
\usepackage{hyperref}
\usepackage[noabbrev]{cleveref}
\usepackage{amssymb}
\usepackage{amsthm}
\usepackage[backend=biber]{biblatex}
\usepackage{tikz}
\usetikzlibrary{cd}
\usepackage{cancel}
\usepackage{multicol}
\usepackage{geometry} 
\usepackage{parskip}
\usepackage{multicol}
\usepackage{xcolor}
\usepackage{ stmaryrd }
\usepackage[polish,english]{babel}
\usepackage{csquotes}

\newtheorem{thm}{Theorem}[subsection]
\newtheorem{lm}[thm]{Lemma}
\newtheorem{claim}[thm]{Claim}
\newtheorem{fact}[thm]{Fact}
\newtheorem{proposition}[thm]{Proposition}
\newtheorem{cor}[thm]{Corollary}

\theoremstyle{definition}
\newtheorem{df}[thm]{Definition}

\newtheorem{remark}[thm]{Remark}

\newtheorem*{acknowledgements}{Acknowledgements}
\renewcommand{\thethm}{%
  \ifnum\value{subsection}=0 
    \thesection
  \else
    \thesubsection
  \fi%
  .\arabic{thm}}

\newcommand{\bbb}{\mathfrak{b}}

\newcommand{\continuum}{\mathfrak{c}}

\newcommand{\Meager}{\mathcal{M}}

\newcommand{\Baire}{\omega^{\omega}}
\newcommand{\Cantor}{2^{\omega}}
\newcommand{\baire}{\omega^{<\omega}}

\newcommand{\forces}{\Vdash}

\newcommand{\cov}{\mathrm{cov}}
\newcommand{\non}{\mathrm{non}}
\newcommand{\add}{\mathrm{add}}
\newcommand{\cof}{\mathrm{cof}}

\newcommand{\IP}{\mathbb{IP}}
\newcommand{\Mminus}{\mathcal{M}_{-}}

\title{On Star operation and some ideals on the Baire space}

\author[A. Cieślak]{Aleksander Cieślak}
\email{aleksander.cieslak@pwr.edu.pl}

\author[D. Perkowska]{Daria Perkowska}
\email{daria.perkowska@pwr.edu.pl}

\author[Sz. Żeberski]{Szymon Żeberski}
\email{szymon.zeberski@pwr.edu.pl}

\address[A. Cieślak, D. Perkowska, Sz. Żeberski]{Faculty of Pure and Applied Mathematics, Wrocław University of Science and Technology, Wybrzeże Stanisława Wyspiańskiego 27, 50-370 Wrocław, Poland}

\date{}
\begin{document}

\begin{abstract}
    In this article we investigate the $\sigma$-ideals $\Mminus$, $\mathcal{ED}$, $\mathcal{UM}$, $\mathcal{P}$, $\mathcal{FM}$ and $\mathcal{H}$ coming from different combinatorial aspects of the $\sigma$-ideal of meager subsets of the Baire space. We calculate their cardinal invariants and investigate the corresponding perfect set style dichotomies. We also investigate the star operation $^*$ related to translations in the Baer-Specker group $\mathbb{Z}^{\omega}$ for these ideals and find combinatorial characterizations of the collections $\Mminus^{*}$, $\mathcal{ED}^{*}$ and $\mathcal{H}^{*}$.
\end{abstract}

\maketitle

\section{Introduction}

We will be investigating subsets of the Cantor space $2^\omega$ and the Baire space $\omega^\omega$, as well as families of such subsets. The basis of the Cantor and Baire  space is formed by clopen sets:
$[\sigma]=\{x\in 2^\omega\text{ or } x\in\omega^\omega:\ \sigma\subseteq x\}$, where $\sigma\in 2^{<\omega}$ or $\sigma\in \omega^{<\omega}$.

We say that a subset $\mathcal{I}\subseteq \mathcal{P}(X)$ is an ideal if $X\notin\mathcal{I}$, $\mathcal{I}$ closed under taking subsets and finite unions. We say that an ideal $\mathcal{I}$ is a $\sigma$-ideal if it is closed under countable unions. Two important $\sigma$-ideals are the family of sets of Lebesgue's measure zero, which we shall denote by $\mathcal{N}$, the family of meager sets denoted by $\mathcal{M}$, and the family of countable sets by $Ctbl$.  By $K_\sigma$ we denote the ideal of countable unions of compact sets in the Baire space $\omega^\omega$ or $\mathbb{Z}^\omega$.

For an ideal $\mathcal{I}$ of Polish space $X$, we define standard cardinal coefficients:
\begin{itemize}
    \item[--] $\add(\mathcal{I})=\min \{|\mathcal{A}|:\ \mathcal{A} \subseteq \mathcal{I} \wedge \bigcup \mathcal{A} \notin \mathcal{I} \},$
    \item[--] $\cov(\mathcal{I})=\min \{|\mathcal{A}|: \mathcal{A} \subseteq \mathcal{I} \wedge \bigcup \mathcal{A} =X \},$
    \item[--] $\non(\mathcal{I})=\min\{|A|: A\notin \mathcal{I}\}$,
    \item[--] $\cof(\mathcal{I})=\min \{|\mathcal{B}|: \mathcal{B} \subseteq \mathcal{I} \wedge (\forall A \in \mathcal{I})(\exists B \in \mathcal{B})(A \subseteq B) \}.$
\end{itemize}
These numbers are known as the \emph{additivity}, the \emph{covering}, the \emph{uniformity} and the \emph{cofinality} of the ideal $\mathcal{I}$.
For $f,g\in \omega^\omega$ (or $\mathbb{Z}^\omega$) let $f<^*g$ mean $\forall^\infty n f(n)<g(n).$
We say that family $F\subseteq\omega^\omega$ is \emph{unbounded} if there is no function $g\in\omega^\omega$ such that for ever $f\in F$ we have $f<^*g.$
We say that family $F\subseteq\omega^\omega$ is \emph{dominating} if for every function $g\in\omega^\omega$ we can find  $f\in F$  and we have $g<^*f.$  As $\exists^\infty n$ we mean that something happens for infinitely many $n$. 
 As $\forall^\infty n$ we mean that something happens for almost all $n$. 

Connected to those concepts we have next cardinal invariants:
\begin{itemize}
    \item[--] $\mathfrak{b}= \min \{ |F| : F \subseteq \omega^\omega \ \wedge F  \text{ is unbounded} \},$
    \item[--] $\mathfrak{d}=\min \{ |F| : F \subseteq \omega^\omega \ \wedge F \text{ is dominating} \}.
$
\end{itemize}

The following diagram, called Cichoń's diagram, summarizes all of the $ZFC$-provable inequalities between cardinal invariants of ideals $\mathcal{M}$, $\mathcal{N}$, $\mathfrak{b}$, and $\mathfrak{d}$.

\begin{center}
\begin{tikzpicture}[node distance=1cm, auto]
    \node (alef1) at (-6, 0) {$\aleph_1$};
    \node (addn) at (-4, 0) {$\add(\mathcal{N})$};
    \node (addm) at (-1, 0) {$\add(\mathcal{M})$};
    \node (covm) at (2, 0) {$\cov(\mathcal{M})$};
    \node (nonn) at (5, 0) {$\non(\mathcal{N})$};
    \node (b) at (-1, 2) {$\mathfrak{b}$};
    \node (d) at (2, 2) {$\mathfrak{d}$};
        \node (covn) at (-4, 4) {$\cov(\mathcal{N})$};
        \node (nonm) at (-1, 4) {$\non(\mathcal{M})$};
                \node (cofm) at (2, 4) {$\cof(\mathcal{M})$};
        \node (cofn) at (5, 4) {$\cof(\mathcal{N})$};
        \node (c) at (7, 4) {$\mathfrak{c}$};
    \draw[->] (alef1) -- (addn);
       \draw[->] (addn) -- (addm);
     \draw[->] (addm) -- (covm);
    \draw[->] (addm) -- (b);
    \draw[->] (addn) -- (covn);
    \draw[->] (covn) -- (nonm); 
    \draw[->] (b) -- (nonm);
    \draw[->] (b) -- (d);
    \draw[->] (nonm) -- (cofm);
    \draw[->] (covm) -- (d);
    \draw[->] (d) -- (cofm);
    \draw[->] (cofm) -- (cofn);
    \draw[->] (covm) -- (nonn);
    \draw[->] (nonn) -- (cofn);
    \draw[->] (cofn) -- (c);

\end{tikzpicture}
\end{center}

In this article, we will need the following presentation of those two invariants using interval partitions.

    By $\mathbb{IP}$ we will denote the set of all partitions $\bar{I}=(I_n)_{n<\omega}$ of $\omega$ into intervals. If $\bar{I}=(I_n)_{n<\omega}$ and $\bar{J}=(J_n)_{n<\omega}$ are interval partitions, then we will write $\bar{I}\sqsubseteq^* \bar{J}$ if for almost every $n$ there exists $k$ such that $I_k\subseteq J_n$.

We say that $F\subseteq\mathbb{IP}$ is \emph{unbounded} if there is no interval partition $\bar{I}\in\mathbb{IP}$ such that for every partition $\bar{J}\in F$ we have $\bar{J}\sqsubseteq^* \bar{I}.$
We say that $F\subseteq\mathbb{IP}$ is \emph{dominating} if for every interval partition $\bar{I}\in\mathbb{IP}$ we can find $\bar{J}\in F$ such that $\bar{I}\sqsubseteq^* \bar{J}.$ 
We will also use another useful representation of invariants $\mathfrak{b}$ and $\mathfrak{d}$ using the interval partitions.\\
Define the bounding and dominating numbers for interval partitions as follows:
let $\mathfrak{b}(\IP)$ be the smallest size of an unbounded family in $\IP$, and let $\mathfrak{d}(\IP)$ be the smallest size of a dominating family in $\IP$. It is well-known that $\mathfrak{b}(\IP)=\mathfrak{b}$ and $\mathfrak{d}(\IP)=\mathfrak{d}$ (see Blass; \cite{Blass2010}).

\section{Ideals on the Baire space}
In this chapter, we will investigate the structural relationships and interactions on the Baire space.

\begin{df}
    Let us define the following ideals on the Baire space $\Baire$:
    \begin{itemize}
        \item A set $X\subseteq\Baire$ is \emph{nowhere dense} if for every $\sigma\in\baire$ there is $\tau\in\baire$ that extends $\sigma$ such that $[\tau]\cap X=\emptyset$. A countable union of nowhere dense sets is called \emph{meager}. Let $\mathcal{M}$ stand for the $\sigma$-ideal of meager subsets of $\Baire$,

        \item A set $X\subseteq \Baire$ is \emph{uniformly nowhere dense} if for every $k\in\omega$ there exists $l\in \omega$ such that for every $\tau\in \omega^k$ there exists $\sigma\in \omega^{k+l}$ such that $\tau\subseteq\sigma$ and $[\sigma]\cap X=\emptyset$. A set $X\subseteq\omega^\omega$ is \emph{uniformly meager} if it is a countable union of uniformly nowhere dense sets. Let $\mathcal{UM}$ stand for the $\sigma$-ideal of uniformly meager sets.
        
        \item A set $X\subseteq \Baire$ is \emph{porous} if there exists $k\in\omega$ such that for every $\sigma \in \omega^{m}$ there is $\tau \in \omega^{m+k}$, $\sigma \subseteq \tau$ and $[\tau]\cap X=\emptyset$. We say that set $X\subseteq\omega^\omega$ is $\sigma$-\emph{porous} if it is a countable union of porous sets. Let $\mathcal{P}$ stands for the $\sigma$-ideal of $\sigma$-porous sets.

        \item For $x\in \Baire$ and an interval partition $\bar{I}$ we define
            \begin{center}
                $M(x,\bar{I})= \{y\in \Baire: \forall^{\infty} n \;\;    x\upharpoonright_{ I_n}\neq y\upharpoonright_{ I_n}\}$  
            \end{center}
        Let $\mathcal{M}_-$ be the $\sigma$-ideal generated by the sets $M(x,\bar{I})$ for $x\in \Baire$ and $\bar{I}\in \mathbb{IP}$. 
    
        \item For $f\in\omega^{\omega}$ and an interval partition $\bar{I}$ we define
        \begin{center}
            $H({f,\bar{I}})=\{x\in\Baire: \forall^{\infty}n $ $\exists m\in I_n \;\; x(m)= f(m)\}$
        \end{center}
        Let $\mathcal{H}$ be $\sigma$-ideal generated by the sets $H({f,\bar{I}})$ where $f\in\omega^{\omega}$ and $\bar{I}\in \mathbb{IP}$.

        \item Let $\mathcal{FM}$ be $\sigma$-ideal $\sigma$-generated by the sets
        \begin{center}
        $F_\phi=\{x\in \Baire:\forall^\infty n \;\; x(n)\neq \phi(x\upharpoonright_n)\}$
        \end{center}
        where $\phi$ ranges over all functions  from $\baire$ to $\omega$.

        \item Let $\mathcal{ED}$ be the eventually different ideal i.e. $\sigma$-ideal $\sigma$-generated by the sets
        \begin{center}
            $E_f=\{x\in\Baire: \forall^{\infty}_n \;\; x(n)\neq f(n)\}$
        \end{center}
        where $f$ ranges over elements of $\Baire$.

    \end{itemize}
\end{df}
All of the considered $\sigma$-ideals are contained (properly) in the $\sigma$-ideal of meager subsets of $\Baire$ and have $F_{\sigma}$-basis.

The ideal $\mathcal{UM}$ was introduced by Hrusak, Wohofsky and Zindulka in \cite{umeager} in the context of Borel Conjecture in arbitrary Polish groups. The authors showed that a Polish group is locally compact if and only if the ideal of meager set and the ideal of uniformly meager sets coincide.

Various types of porous sets were studied by researchers in many different contexts. This $\sigma$-ideal was investigated by Zindulka and Hru{\v{s}}{\'a}k in  \cite{hruvsak2012cardinal} and by Arturo Martinez-Celis and Osvaldo Guzman in \cite{GuzmanMartinezStronglyPorous}. See \cite{GuzmanMartinezStronglyPorous}, \cite{hruvsak2012cardinal}, \cite{RepickyCardInvsPorous}, or \cite{BrendleCardInvsPorosity} for cardinal invariants of porous sets, or \cite{MR2076859}, \cite{MR2025819} for descriptive set theoretic properties of such sets.

The $\sigma$-ideal $\mathcal{M}_{-}$ was introduced in \cite{EEEE} by Mazurkiewicz, Michalski, Rałowski, and Żeberski as an adaptation of chopped reals (see characterization of meager sets Theorem $2.2.4$ in \cite{BBB}) to the Baire space. In \cite{EEEE} the authors showed that $\mathcal{M}_{-}$ is properly contained in the $\sigma$-ideal of meager sets.

The $\sigma$-ideal $\mathcal{H}$ was implicitly used by Bartoszyński in \cite{BartoszCombAspects} (see Theorem $2.2.12$ in \cite{BBB}) to characterize $\add(\mathcal{M})$ and $\cof(\mathcal{M})$.  This ideal $\mathcal{H}$ also appears in the study of meager-additive sets. Clearly, every set $H({f,\bar{I}})$ is meager.  Note also that for every sequence $\{\bar{I}^n:n<\omega\}$ of interval partitions and every sequence $\{f_n:n<\omega\}\subseteq\omega^\omega$ there are $\bar{J}\in \mathbb{IP}$ and $g\in \omega^\omega$ such that $\bigcup_{n} H({f_n,\bar{I}^n})\subseteq H({g, \bar{J}})$.

The $\sigma$-ideal $\mathcal{FM}$ was introduced by Rosłanowski and Newelski in \cite{MR1112500}, where the authors calculated its cardinal invariants and investigated forcing properties of trees corresponding to the $\sigma$-ideal by perfect-set style dichotomy. Later, the $\sigma$-ideal was investigated by Khomskii and Laguzzi in \cite{KHOMSKII} in the context of measurability and the half-Cohen problem.

The eventually different ideal $\mathcal{ED}$ was studied in \cite{KHOMSKII} and \cite{SpinasPerfSetTheorems}, where the tree dichotomy for this ideal was established. The cardinal invariants of $\mathcal{ED}$ were studied by Rosłanowski and Newelski in \cite{MR1112500} and by Khomskii and Laguzzi in \cite{KHOMSKII}.

Now, we will introduce a characterization of uniformly meager sets in the spirit of chopped reals, that we will frequently use later on.
\begin{lm}\label{UM}
    A set $X\subseteq\omega^\omega$ is uniformly meager if and only if there exists interval partition $\bar{I}\in \mathbb{IP}$ and a function $\psi:\bigcup_n\omega^{\min(I_{n})}\rightarrow \bigcup_n\omega^{I_n} $ such that $dom( \psi(\sigma))= I_n$, for $\sigma\in \omega^{\min(I_{n})}$
    and $$X\subseteq U(\psi,\bar{I})=\{x\in\Baire: \forall^\infty n \;\; x\upharpoonright_{ I_n} \neq \psi (x\upharpoonright_{ \min(I_{n})}) \}.$$
\end{lm}
\begin{proof}
    $(\Longleftarrow)$ As we clearly have
    \begin{center}
        $U(\psi, \bar{I})=\bigcup_m \{x: \forall\; n\geq m \;\; x\upharpoonright_{ I_n} \neq \psi (x\upharpoonright_{ \min(I_{n})})\}$,
    \end{center}
    it is enough to show that the set $U_{m}=\{x: \forall\; n\geq m \;\; x\upharpoonright_{ I_n} \neq \psi (x\upharpoonright_{ \min(I_{n})}) \}$ is a uniformly nowhere dense set. Suppose that $l\in\omega$ and $\sigma\in\omega^{l}$ are given. Let $m\leq n$ be the smallest such that $|\sigma|<\min I_n$. Define 
    \begin{center}
        $\sigma'=\sigma^\frown{{\underbrace{ 0^\frown 0^\frown\ldots^\frown 0}_{\min(I_{n})-|\sigma|}}}.$
    \end{center}
    Now $\tau={\sigma'}^{\frown}\psi(\sigma')$ is such that  $\tau\supseteq\sigma$ and $Y\cap [\tau]=\emptyset$. Note also that the length of such $\tau$ is the same for every $\sigma\in\baire$ of length $l$.

     $(\Longrightarrow)$ Now we will show that for every uniformly meager set $X\subseteq\Baire$ there are $\psi$ and $\bar{I}$ such that $X\subseteq U(\psi, \bar{I})$. Suppose that $X$ is an increasing union of the sequence $\{X_k:n<\omega\}$ of uniformly nowhere dense sets. For any $X_k$ we have $X_k\subseteq \{x: \forall\; n \;\; x\upharpoonright _{I^k_n} \neq \psi_k (x\upharpoonright_{ \min(I^k_{n})}) \}$ and WLOG $ \min(I_n^k)\leq \min(I_n^{k+1}).$

     We will construct $\bar{I}\in \mathbb{IP}$ and $\psi:\bigcup_n\omega^{\min(I_{n})}\rightarrow \bigcup_n\omega^{I_n} $ such that $X\subseteq U(\psi, \bar{I})$:
     
     \begin{itemize}
         \item[--] Let $\min(I_n)=\min(I_n^n)$ ,
         \item[--] let $\psi(y\upharpoonright _{\min(I_n)})= \psi_n(y\upharpoonright _{\min(I_n)})^\frown{{\underbrace{ 0^\frown 0^\frown\ldots^\frown 0}_{\max(I_{n+1})-\max(I_{n})}}}$. 
         
     \end{itemize}
     
     Clearly, according to the construction, we have $X\subseteq U(\psi, \bar{I})$.  
\end{proof}

\begin{lm}\label{umlema}
    For $\bar{I}, \bar{J}\in \mathbb{IP}$ and $\psi:\bigcup_n\omega^{\min(I_{n})}\rightarrow \bigcup_n\omega^{I_n}$, $\phi:\bigcup_n\omega^{\min(J_{n})}\rightarrow \bigcup_n\omega^{J_n}$ we have $U(\psi, \bar{I})\subseteq U(\phi ,\bar{J})\Longrightarrow \bar{I}\sqsubseteq^*\bar{J}$
\end{lm}
\begin{proof}
Assume that we have $I\nsubseteq ^*J$. Thus there is an infinite set $A\subseteq\omega$ such that for every $n\in A$ there exists $k_n$ such that  $ J_n\subseteq I_{k_n}\cup I_{k_n+1}$. It is not difficult to inductively construct $x\in \Baire$ such that for every $n\in \omega$, if $n\in A$:
\begin{enumerate}
    \item Let $x\upharpoonright_{ (I_{k_n}\setminus J_n)} \neq\psi(x\upharpoonright_{ min(I_{k_n})})\upharpoonright_{ (I_{k_n}\setminus J_n)}$,
    \item let $x\upharpoonright_{ J_n}= \phi (x\upharpoonright_{ min (J_n)})$,
    \item let $x\upharpoonright_{ (I_{k_n+1}\setminus J_n)}\neq \psi (x\upharpoonright_{ min I_{k_n+1}})\upharpoonright _{(I_{k_n+1}\setminus J_n)}$,
\end{enumerate} 

and if $n\notin A$, 

\begin{enumerate}
    \item [(4)]
    let $x\upharpoonright_{ I_n}\neq \psi(x\upharpoonright _{min(I_n)})$
\end{enumerate}

Of course, $x\in  U(\psi, \bar{I})$ because of $(1),(3)$, $(4)$ , and $x\notin U(\phi ,\bar{J})$ because of $(2).$
    
\end{proof}

The following lemma gathers basic properties of the basis of the $\sigma$-ideals $\mathcal{M}_-$ and $\mathcal{H}$ that we will need throughout the article.
\begin{lm}\label{XD}\label{sssss}
    For all $x,y,z\in\Baire$ and $\bar{I},\bar{J}\in \mathbb{IP}$ the following hold:
    \begin{itemize}
        \item  [--] if $I\sqsubseteq^* J$, then $M(x,\bar{I})\subseteq M(x,\bar{J})$,
        \item [--] if $M(y,\bar{I})\subseteq M(z,\bar{J})$, then $  \bar{I}\sqsubseteq^*\bar{J}$,
        \end{itemize}
        also
        \begin{itemize}
        \item [--] if $\bar{I}\sqsubseteq^*\bar{J}$, then $H({f,\bar{I}})\subseteq H({f,\bar{J}})$,
        \item[--]  if $H({f, \bar{I}})\subseteq H(x,  {\bar{J}})$, then $\bar{I}\sqsubseteq^*\bar{J}$
    \end{itemize}
\end{lm}
\begin{proof}
    The same as the proof for chopped reals in \cite{BBB}.
\end{proof}

\begin{proof} We will only prove the last observation.
    Assume that we have $I\nsubseteq ^*J$, so for infinitely many $n$ there exists $k$ such that  $ J_n\subseteq I_k\cup I_{k+1}$. 
    Let $\{R_m: m\in \omega\}$ be the increasing enumeration of those $J_n$.
    Then construct $y\in \omega^\omega$ as follows:
    \begin{enumerate}
        \item $y(i)=x(i)+1$ if $i\in R_m$ for some even  $m$.
         \item $y(i)=f(i) $ if $i\in R_m$ for some odd $m.$
         \item $y(i)=f(i) $ if $i\notin R_m$ for any $m.$
    \end{enumerate}
Of course, $y\in H({f, \bar{I}})$ because of $(2)$ and $(3)$. But $y\notin H({x, \bar{J}})$ because of $(1)$.
\end{proof}

\begin{thm}\label{kkkkkkk}
    The following inclusions holds and are proper. Also, no other inclusion holds.

    \begin{center}
\begin{tikzpicture}[node distance=1.5cm, auto]
    \node (M) at (0, 4.5) {$\mathcal{M}$};
    \node (MM) at (0, 3) {$\mathcal{UM}$};
    \node (Mminus) at (0, 1.5) {$\mathcal{M}_-$};

    \node (P) at (2, 1.5) 
    {$\mathcal{P}$};
    \node (H) at (-1, -0.5) {$\mathcal{H}$};
    \node (ED) at (1, -0.5) {$\mathcal{ED}$};
        \node (FM) at (1.5, 0.5) {$\mathcal{FM}$};

   
    \draw[->] (H) -- (Mminus);
        \draw[->] (ED) -- (FM);
              \draw[->] (FM) -- (P);
            \draw[->] (P) -- (MM);
    \draw[->] (ED) -- (Mminus);
  
    \draw[->] (Mminus) -- (MM);
    \draw[->] (MM) -- (M);
\end{tikzpicture}
\end{center}

\end{thm}
\begin{proof}
    \begin{itemize}

        \item The ideal $\mathcal{ED}$ is properly contained in the ideal $\mathcal{M}_{-}$.

           Take any set $E_f$. Then  for any partition $\bar{I}$ we get  $E_f\subseteq M(f, \bar{I})\in \mathcal{M}_-.$

    Now we will show that the inclusion is proper. Let $X=M(0,\bar{I})$ where $\bar{I}$ is a partition into intervals such that $|I_n|=n+2$, say $I_n=\{i_0^n,_{\cdots}, i_{n+1}^n\}$. Assume there exists $(f_n)_{n<\omega}\subseteq \omega^\omega$ such that  $X\subseteq\bigcup_n E_{f_n}.$

    We will construct $x\in X $ such that $x\notin \bigcup_n E_{f_n}.$ Let:
    \begin{enumerate}
        \item $x(i_j^n)=f_j(i_j^n)$ for $j\leq n$,
        \item $x(i_{n+1}^n)=1$. 
    \end{enumerate}
    Of course we have $x\in X$ because of $(2)$ and $x\notin \bigcup_n E_{f_n}$ because of $(1).$

\item $\mathcal{FM} $ is properly contained in $\mathcal{P}$ 

Take any $X\in \mathcal{FM}$. We know that for some $\psi: \omega^{<\omega}\rightarrow \omega $ we have    $X\subseteq \{x\in \omega^\omega: \forall^\infty n \;\;\;\;x(n)\neq \psi (x\upharpoonright_n)\}= \bigcup_N \{x\in \omega^\omega: \forall m>N\;\;\;\;x(n)\neq \psi (x\upharpoonright_n)\}$. Of course every set of form
$ \{x\in \omega^\omega: \forall m>N\;\;\;\;x(n)\neq \psi (x\upharpoonright_n)\}$ is $N$-porous.

Now we will show that the inclusion is proper. Take tree $T\subseteq \omega^\omega$ such that: \begin{enumerate}
    \item $|succ(\sigma)|=1$, when $|\sigma|=n$ for even $n.$
    \item $succ(\sigma)=\omega$, when $|\sigma|=n$ for odd $n.$
\end{enumerate}

Of course, $T$ is a $2$ porous tree. Take any $\sigma\in \omega^n$. Because of $(1)$ we can find $\beta\supseteq\alpha$ and $\beta\in \omega^{n+2}$ such that $[\beta]\cap T=\emptyset.$ Of course $T\notin \mathcal{FM} $ because of $(2).$

\item The ideal $\mathcal{UM}$ is properly contained in $\mathcal{M}$.

It is easy to see that every uniformly nowhere dense set is nowhere dense. So we have $\mathcal{UM}\subseteq \mathcal{M}$. The inclusion is proper, as shown in \cite{umeager}

    \item The ideal $\mathcal{M}_{-}$ is not contained in the ideal $\mathcal{P}$. 

     We will show that the set $M=\{x\in\Baire:\forall n\in\omega$ $x|_{I_{n}}\neq0|_{I_{n}}\}\in\mathcal{M}_{-}$, where $(I_{n})_{n\in\omega}$ is an interval partition made of intervals of increasing length, is not in $\notin\mathcal{P}$. Assume that $M\subseteq \bigcup_{n}P_{n}$ where every $P_{n}$ is $k_{n}$-porous. On the one hand, by Baire category theorem, there is $\sigma\in\omega^{<\omega}$ and $n\in\omega$ such that $M\cap[\sigma]\subseteq P_{n}$. On the other hand, it is not difficult to construct $y\in (M\cap[\sigma])\setminus P_{n}$.

     \item      The ideal $\mathcal{M}_-$ is properly contained in the ideal $\mathcal{UM}$.

 First, we will show that $\mathcal{M}_-\subseteq\mathcal{UM}$
    Take any $M(x,\bar{I})\in \mathcal{M}_-$. For $\sigma\in \omega^{\min(I_{n})}$ and $n\in \omega$, set $\psi (\sigma)=x\upharpoonright _{I_n}.$ Of course, $M(x,\bar{I})\subseteq U(\psi, \bar{I})$.

Now we will show that this inclusion is proper.

Let $\psi$ be a $1-1$ and onto function and let $\bar{J}$ be a partition into points. Take $X=U(\psi,\bar{J})=\{x:\forall^\infty n \;\; x(n)\neq \psi(x\upharpoonright_{ n})\}$. We will show that $X\nsubseteq M(y,\bar{I})$ for any $x\in \mathbb{Z}^\omega $ and $I\in \mathbb{IP}.$ We will construct $z\in X\backslash M(y,\bar{I})$.

Let $z\in \mathbb{Z}^{\bigcup_{i<n}I_{2i}\cup I_{2i+1}}$ be such that:
\begin{itemize}
    \item [--] $z(j)\neq\psi(z\upharpoonright_{ j})$ for $j\in I_{2n}\cup I_{2n+1}$,
    \item  [--] $z\upharpoonright _{I_{2n+1}}=y\upharpoonright _{I_{2n+1}}.$
\end{itemize}

We can construct such $z$, because for every $i\in I_{2n}\cup I_{2n+1}$ there exist finitely many $\tau_i\in \omega^{\max I_{2n+1}}$ such that: 
\begin{itemize}
   \item [--] $z \upharpoonright_{   \min I_{2n}}\subseteq \tau_i, $
    \item  [--] $\tau_i (j)=\psi (y\upharpoonright_{ j})$ for $j\in I_{2n+1}$.
\end{itemize}
    \item $\mathcal{ED}$ is properly contained in $\mathcal{FM}$ 

Take any $E_f\in \mathcal{ED}.$ Of course for $\psi: \omega^{<\omega}\rightarrow\omega$ such that for every $x\in \omega^\omega$,  $n\in \omega$ and $\psi(x\upharpoonright_n)=f(n)$ we have $E_f \subseteq \{x\in \omega^\omega: \forall^\infty n \;\;\;\;x(n)\neq \psi (x\upharpoonright_n)\}$. So $E_f \in \mathcal{FM}$.

The inclusion is proper, because set $X\in \mathcal{UM}\setminus \mathcal{M}_-$ that we constructed in previous lemma is in $\mathcal{FM}$. So because $X\notin \mathcal{M}_-$ we have $X\notin \mathcal{ED}$

    \item    The ideal $\mathcal{H}$ is properly contained in $\mathcal{M}_-$.

       For any $H({f, \bar{I}})\in \mathcal{H}$ we have $H({f,\bar{I}})\subseteq M(f+1,\bar{I}).$

Now we will show that this inclusion is proper. 
Let $X=M(0,\bar{I})$ where $\bar{I}$ is a partition into singletons. Assume there exists $\bar{J}\in\mathbb{IP},f\in \omega^\omega$ such that $M(0,\bar{I})\subseteq H({f,\bar{J}})$. We will construct $x\in \omega^\omega$ as follows:
\begin{itemize}
    \item [--]  If $f(n)\neq 1$, let $x(n)=f(n)-1$;
    \item [--]  if $f(n)=1$, let $x(n)=2$.

\end{itemize}

Of course, $x\in M(0,\bar{I})$ and $x\notin H({f,\bar{J}}).$

    \end{itemize}

\end{proof}

\section{Cardinal invariants}
In this chapter, we will show that cardinal invariants of $\mathcal{M}_-$ are equal to the corresponding cardinal invariants of a $\sigma$-ideal of meager sets. For this we will need the following well-known combinatorial characterizations of covering and uniformity of meager sets.\\
We say that family $F\subseteq  \omega^\omega$ is \emph{infinitely equal} if for every $g\in\omega^\omega$ we can find $f\in F$ such that $\exists^\infty n $ such that $f(n)=g(n).$ We say that family $F\subseteq \omega^\omega$ is \emph{eventually different} if for any $g\in \omega^\omega$ we can find $f\in F$ such that $\forall^\infty n $ we have $f(n)\neq g(n).$ 
We recall the combinatorial characterizations of the cardinal coefficients $\non(\mathcal{M})$ and $\cov(\mathcal{M})$ in terms of functions in the Baire space.

\begin{thm}(Bartoszyński, Miller \cite{BBB})\label{xd}
    \begin{itemize}
        \item [--] $\non(\mathcal{M})=\min\{|F|: F\subseteq \omega^\omega \wedge F \text{ is an infinitely equal family}\}$,
        \item [--] $\cov(\mathcal{M})=\min\{|F|: F\subseteq \omega^\omega \wedge F \text{ is an eventually different family}\}$.
    \end{itemize}

\end{thm}
First we calculate the covering and uniformity of $\sigma$-ideals $\mathcal{UM}$ and $\Mminus$.
\begin{proposition}\label{CovNonMminusAndUM}
The following equalities hold:
\begin{itemize}
    \item[--] $\cov(\mathcal{UM})=\cov(\mathcal{M}_-)=\cov(\mathcal{M})$,
    \item[--] $\non(\mathcal{UM})=\non(\mathcal{M}_-)=\non(\mathcal{M})$.
\end{itemize}
\end{proposition}
\begin{proof}
    Since by Theorem \ref{kkkkkkk} $    \mathcal{M}_-\subseteq\mathcal{UM}\subseteq\mathcal{M}$ it will be enough to check that $\cov(\mathcal{M}_-)\leq \cov(\mathcal{M})$ and $\non(\mathcal{M}_-)\geq \non (\mathcal{M})$.\\
    To prove the inequality  $\cov(\mathcal{M})\leq\cov (\mathcal{M}_-)$we will use characterization of $\cov(\mathcal{M)} $ from \Cref{xd}. Let family $F=\{x_\alpha\in \omega^\omega :  \alpha<\cov(\mathcal{M})\}$ be an eventually different family. Take any interval partition $\bar{I}$ and let $M_\alpha=M(x_\alpha,\bar{I})$. We will show that $\bigcup M_\alpha =\omega^\omega$.
    Take any $y\in\omega^\omega.$ Since family $F$ is of size $\cov(\mathcal{M})$ we can find $x_\alpha \in F$ such that $x_\alpha\neq^*y$. It follows that $y\in M(x_\alpha,I)$.\\
    To prove the other inequality take set $\{x_\alpha: \alpha < \non(\mathcal{M}_-)\}\notin M(y, \bar{I})$ for every $y\in \omega^\omega$ and any interval partition $\bar{I}.$ 
    So for every $y\in\omega^\omega$ there exists $x_\alpha$ such that $x_\alpha=^*y$. So the family $\{x_\alpha: \alpha<\non(\mathcal{M}_-)\}$ is exactly the family from \Cref{xd}.
\end{proof}
In \cite{MR1112500} and \cite{KHOMSKII} the authors calculated the cardinal invariants of the $\sigma$-ideals $\mathcal{FM}$ and $\mathcal{ED}$.
\begin{thm}(Rosłanowski and Newelski \cite{MR1112500}; Khomskii and Laguzzi \cite{KHOMSKII})
The following hold:
\begin{itemize}
    \item[--] $\add(\mathcal{FM})= \add(\mathcal{ED})=\omega_{1}$,
    \item[--] $\cov(\mathcal{FM})= \cov(\mathcal{ED})=\cov(\mathcal{M})$,
    \item[--] $\non(\mathcal{FM})=\non(\mathcal{ED})= \add(\mathcal{M})$,
    \item[--] $\cof(\mathcal{FM})= \cof(\mathcal{ED})=\continuum$.
\end{itemize}
\end{thm}
Regarding the additivity and cofinality of the $\Mminus$ ideal we have the following result.
\begin{thm}
The following hold:
\begin{itemize}
    \item[--] $\add(\mathcal{M}_-)= \add(\mathcal{M})$,
    \item[--] $\cof(\mathcal{M_-})= \cof(\mathcal{M})$.
\end{itemize}
\end{thm} 
\begin{proof}
Recall that $\add(\mathcal{M})=\min\{\mathfrak{b},\cov(\mathcal{M})\}$.
First, we show that
$\add(\mathcal{M}_-)\leq \mathfrak{b}$. Assume that $\bar{I}^\alpha$ is an unbounded family of interval partitions. Then $\bigcup_{\alpha<\bbb}M(0,\bar{I}^\alpha)$ is not in $\Mminus$. Because if for some $x\in\Baire$ and $J\in\IP$ we have $\bigcup_{\alpha<\bbb}M(0,\bar{I}^\alpha)\subseteq M(x,J)$, then by Proposition \ref{sssss}, $J$ dominates all $I^{\alpha}$'s.\\
                 
The inequality $\add(\mathcal{M}_-) \leq \cov(\mathcal{M})$ follows from Proposition \ref{CovNonMminusAndUM}.\\
     
To prove the inequality $\add(\Meager)\leq\add(\Mminus)$ assume that $\kappa<\add(\mathcal{M})$ and we are given $\{x_{\alpha}:\alpha<\kappa\}\subseteq\Baire$ and a collection $\{\bar{I}^\alpha:\alpha<\kappa\}\subseteq\IP$. Recall that $\add(\mathcal{M})=\min\{\mathfrak{b}, \cov(\mathcal{M})\}$. Since $\kappa<\cov(\mathcal{M})$ there is $z\in\Baire$ such that for every  $\forall\alpha<\kappa$ there are infinitely many $n\in\omega$ such that $z\upharpoonright_{ I_n^\alpha}=x_\alpha \upharpoonright_{ I_n^\alpha}$. For every $\alpha<\kappa$ pick an interval partition $\bar{J}^\alpha$ such that $\forall n\; \exists k\;(I_k^\alpha \subseteq J_n^\alpha \wedge z\upharpoonright_{ I_k^\alpha}= x_\alpha\upharpoonright_{ I_n^\alpha})$. Since $\kappa<\mathfrak{b}$ there is $\bar{R}\in\mathbb{IP}$ such that $\forall\alpha \;\; \bar{J}^\alpha \sqsubseteq^* \bar{R}$. It follows that $\bigcup_{\alpha<\kappa} M_\alpha \subseteq M(z,\bar{R})$.\\

The inequality $\cof(\mathcal{M}_-)\geq \non(\mathcal{M})$ follows from Proposition \ref{CovNonMminusAndUM}.

Next, we will show $\cof(\mathcal{M}_-)\geq \mathfrak{d}$.

Let $\{M(x_\alpha, \bar{I}^\alpha):\alpha< \cof (\mathcal{M}_-)\}$ be a basis of $\mathcal{M}_-$. 
       We claim that $\{\bar{I}^\alpha:\alpha < \cof(\mathcal{M}_-)\}$ forms a dominating family. For a partition $\bar{J}$, there exists $\alpha$ such that $M(0,\bar{J})\subseteq M(x_\alpha, \bar{I}^\alpha)$. Thus by Theorem  \ref{XD} we have $\bar{J}\sqsubseteq^*\bar{I}^\alpha.$   

        Now we will prove that $\cof(\mathcal{M}_-)\leq \cof (\mathcal{M})$.

        Let $\{\bar{I}^\alpha: \alpha<\mathfrak{d}\}$ be a dominating family in $\mathbb{IP}$ and for $\alpha<\mathfrak{d}$ let $\langle f^\alpha_\gamma: \gamma<\non(\mathcal{M})\rangle$ be such that for every $f\in \omega^\omega$ there exists $\gamma$ such that for infinitely many $n$ we have $f\upharpoonright_{ I_n^\alpha}= f_\gamma^\alpha\upharpoonright_{ I_n^\alpha}$. Let $M^{\alpha, \beta}_\gamma= M(\bar{I}^\beta, f^\alpha_\gamma).$ 
      We will finish the proof by showing that 
            $\{ M^{\alpha, \beta}_\gamma: \alpha, \beta<\mathfrak{d}, \gamma<\non(\mathcal{M})\}$ is a basis of the ideal $\mathcal{M}_-$.  
            
        Suppose we are given any $M(x, \bar{J})$. Of course, we can find $\alpha$ such that $M(x,\bar{J})\subseteq M(x, \bar{I}^\alpha)$. There exists $f^\alpha_\gamma\in \Baire$ such that for infinitely many $n$ we have $x\upharpoonright_{ I_n^\alpha}=f_\gamma^\alpha\upharpoonright_{ I_n^\alpha}$.  Let $\bar{R}$ be an interval partition such that for almost all $n$ there exists $k$ such that $I^\alpha_k\subseteq R_n$ and $x\upharpoonright_{ I_k^\alpha} =f^\alpha_\gamma\upharpoonright_{ I_k^\alpha}.$ Let $\bar{I}^\beta\in \IP$ be such that $R\sqsubseteq^* I^\beta$. Of course $M(x,\bar{J})\subseteq M(f^\alpha_\gamma
        , I^\beta)$.
\end{proof}
Now we will move to the cardinal invariants of the ideal $\mathcal{UM}$. By Proposition \ref{CovNonMminusAndUM} we already know that $\cov(\mathcal{UM})=\cov(\Meager)$ and $\non(\mathcal{UM})=\non(\Meager)$. The additivity and cofinality of uniformly meager sets are as follows.

\begin{thm}
    The following hold:
    \begin{itemize}
        \item[--] $\add(\mathcal{UM})=\omega_{1}$,
        \item[--] $\cof(\mathcal{UM})=\continuum$.
    \end{itemize}
\end{thm}
\begin{proof}
    We first handle the additivity of uniformly meager sets. Let $(I_{n})_{n\in\omega}$ be any interval partition of $\omega$ into intervals that are of the same length and let $\{\phi_{\alpha}:\alpha<\continuum\}$ be such that:
    \begin{itemize}
        \item[--] for every $\alpha<\continuum$ $\phi_{\alpha}:\bigcup_{n}\omega^{\min I_{n}}\rightarrow\baire$ is such that $\phi_{\alpha}(\sigma)\in\omega^{I_{n}}$ for $\sigma\in\omega^{\min I_{n}}$,
        \item[--] for $\alpha,\beta<\continuum$, if $\alpha\neq\beta$, then for all but finitely many $\sigma$, we have $\phi_{\alpha}(\sigma)\neq\phi_{\beta}(\sigma)$.
    \end{itemize}
    For every $\alpha<\continuum$, let $T_{\alpha}\subseteq\baire$ be such a tree that $[T_{\alpha}]=\{x\in\Baire:\forall n\in\omega$ $\phi_{\alpha}(x\upharpoonright_{\min I_{n}})\neq x\upharpoonright_{I_{n}}\}$. We will show that $\bigcup_{\alpha<\omega_{1}}[T_{\alpha}]\notin\mathcal{UM}$. Towards contradiction, suppose that $\bigcup_{\alpha<\omega_{1}}[T_{\alpha}]\subseteq U(\psi,(J_{n})_{n\in\omega})$ (see lemma \ref{UM}) for some interval partition $(J_{n})_{n\in\omega}$ and $\psi:\bigcup_{n}\omega^{\min J_{n}}\rightarrow\baire$ satisfying $\psi(\sigma)\in\omega^{J_{n}}$ for each $\sigma\in\omega^{\min J_{n}}$. By Lemma \ref{umlema}, we may assume that $(J_{n})_{n\in\omega}$ dominates $(I_{n})_{n\in\omega}$. We will need the following claim:
    \begin{claim}
        There is $\alpha<\omega_{1}$ such that for every $\sigma\in T_{\alpha}$ there is $\tau$ extending $\sigma$ and $m^*\in\omega$ such that:
        \begin{itemize}
            \item[--] $|\tau|=\max (J_{m^{*}})$ and $J_{m^*}\subseteq|\tau|\setminus|\sigma|$
            \item[--] if $I_{n}\subseteq J_{m^*}$, then $\tau\upharpoonright_{I_{n}}=\phi_{\alpha}(\tau\upharpoonright_{(\min I_{n})})$
            \item[--] $\tau\upharpoonright_{J_{m^*}}=\psi(\tau\upharpoonright_{\min (J_{m^*})})$.
        \end{itemize}
        \end{claim}
        Notice that once the claim is proven, if $\alpha<\omega_{1}$ is as stated in the claim, by induction we can easily construct sequences $\{\tau_{n}:n\in\omega\}$ and $\{m_{n}:n\in\omega\}$ such that for every $n\in\omega$ we have: 
        \begin{itemize}
            \item[--] $m_{n}>m_{n-1}$
            \item[--] $\tau_{n}\in T_{\alpha}\upharpoonright_{\tau_{n-1}}$
            \item[--] $\tau_{n}\upharpoonright_{J_{m_{n}}}=\psi(\tau_{n}\upharpoonright_{\min (J_{m_{n}})})$
        \end{itemize}
        Then, for $x=\bigcup_{n}\tau_{n}$, we have that $x\in[T_{n}]\setminus M(\psi,(J_{m})_{m\in\omega})$, which contradicts the previous assumption.\\
        To prove the claim, assume towards a contradiction that for every $\alpha<\omega_{1}$ there is $\sigma_{\alpha}\in T_{\alpha}$ such that for every $\tau$ extending $\sigma_{\alpha}$, if $m^{*}\in\omega$ is such that $\max (J_{m^*})=|t|$, $J_{m^*}\subseteq|\tau|\setminus|\sigma_{\alpha}|$, and $\tau\upharpoonright_{J_{m^*}}=\psi(\tau\upharpoonright_{\min (J_{m^*})})$, there is $I_{n}\subseteq J_{m^{*}}$ with $\tau\upharpoonright_{I_{n}}=\phi_{\alpha}(\tau\upharpoonright_{\min (I_{n})})$.
        Without loss of generality we may assume that all $\sigma_{\alpha}$'s are equal to some fixed $\sigma\in\baire$. Choose $m^*\in\omega$ such that $J_{m^*}\subseteq\omega\setminus|\sigma|$ and assume that $N$ is the number of intervals $I_{n}$ contained in $J_{m^*}$. Choose $N+1$ distinct elements of $\omega_{1}$, say $\{\alpha_{0},...,\alpha_{N}\}$. Pick now $\tau$ extending $\sigma$ that satisfies the following properties:
        \begin{itemize}
            \item[--] $|\tau|=\max (J_{m^*})$,
            \item[--] $\psi(\tau\upharpoonright_{\min (J_{m^*})})=\tau\upharpoonright_{J_{m*}}$,
            \item[--] for every $n$ such that $I_{n}\subseteq\omega\setminus|\sigma|$, $\rho\in \omega^{\min (I_{n})}$ and distinct $i,j\in N+1$ we have $\phi_{\alpha_{i}}(\sigma)\neq\phi_{\alpha_{j}}(\sigma)$.
        \end{itemize}
        Finding such $\tau$ is possible as functions $\phi_{\alpha}$ are eventually different. By the pigeon-hole principle, there is $I_{n}$ such that for two distinct $i,j\in N+1$ such that $\phi_{\alpha_{i}}(\tau\upharpoonright_{\min (I_{n})})=\tau\upharpoonright_{I_{n}}=\phi_{\alpha_{j}}(\tau\upharpoonright_{\min (I_{n})})$, which is the final contradiction.\\

        To see that $\cof(\mathcal{UM})=\continuum$, assume that $\kappa<\continuum$ and a family $\{U_{\beta}:\beta<\kappa\}\subseteq\mathcal{UM}$ forms a basis of $\mathcal{UM}$. Then there must be an $\beta<\kappa$ such that $U_{\beta}$ covers $\omega_{1}$-many sets $[T_{\alpha}]$, which we already know cannot be true.
\end{proof}

Now we will move to the cardinal invariants of the ideal $\mathcal{H}$. We will calculate the additivity and cofinality of the $\sigma$-ideal $\mathcal{H}$.
\begin{thm}
The following holds:
\begin{itemize}
    \item[--] $\add(\mathcal{H})=\non(\mathcal{H})=\add(\mathcal{M}),$
    \item[--] $\cof(\mathcal{H})=\cov(\mathcal{H})=\cof(\mathcal{M}).$
\end{itemize}
\end{thm}
\begin{proof}
In \cite{BartoszCombAspects} Bartoszyński showed that $\non(\mathcal{H})=\add(\mathcal{M})$ and  $\cov(\mathcal{H})=\cof(\mathcal{M})$ (see also \cite{BBB}), so it will be sufficient to show that $\add(\Meager)\leq\add(\mathcal{H})$ and $\cof(\mathcal{H})\leq\cof(\Meager)$.\\
To show that $\add(\Meager)\leq\add(\mathcal{H})$, assume we have $\add(\mathcal{H})<\kappa<\add(\mathcal{M})$. Take $\{H^\alpha: \alpha<\kappa\}$, where $H^\alpha=H({x_\alpha , \bar{I^\alpha}})$.
We will show that $\bigcup_{\alpha<\kappa} H^\alpha\subseteq \mathcal{H}$. For every $\alpha$, we will construct $z_\alpha\in\omega^\omega$ such that $z_\alpha(n)=\text{"} x_\alpha \upharpoonright_{ I^\alpha_n} \text{"}$.
Since $\kappa<\add(\mathcal{M})$, we can find $c$ such that for all $\alpha$, there exist infinitely many $n$ such that $z_\alpha(n)=c(n)$.
Now we want to construct $d$ in such a way that if $z_\alpha(n)=c(n)$, let $d\upharpoonright_{ I^\alpha_n}=x_\alpha\upharpoonright _{I^\alpha_n}.$ If $z_\alpha(n)\neq c(n)$, let $d\upharpoonright _{I^\alpha_n}=0$.
Next, let $I^\alpha\sqsubseteq^* J^\alpha$ such that for all $n$ there exists $k$ such that $I^\alpha_k \subseteq J^\alpha_n$ and $x_\alpha \upharpoonright_{ I^\alpha_k} =d \upharpoonright_{ I_k^\alpha}.$ 
Now because $\kappa<\add(\mathcal{M)}$ there is $R\in \mathbb{IP}$ such that for every $\alpha$ $J^\alpha\sqsubseteq^*R$. But then we have that $\bigcup_{\alpha<\kappa} H({x_\alpha, I^\alpha}) \subseteq H({d,R})$.\\

We will now prove the inequality $\cof(\mathcal{H})\leq \cof(\mathcal{M})$. For this, we will use the characterization $\cof(\mathcal{M})=\min\{\mathfrak{d},\non(\mathcal{M})\}$.

        Assume that $\langle\bar{I}^\alpha: \alpha<\mathfrak{d}\rangle$ is a dominating family in $\mathbb{IP}$ and for $\alpha<\mathfrak{d}$, let $\langle f^\alpha_\gamma: \gamma<\non(\mathcal{M})\rangle$ be such that for every $f\in \omega^\omega$ there exists $\gamma$ such that for infinitely many $n$ we have $f\upharpoonright_{ I_n^\alpha}= f_\gamma^\alpha\upharpoonright _{I_n^\alpha}$. Let $H^{\alpha, \beta}_\gamma= H(\bar{I}^\beta, f^\alpha_\gamma)$. 
        
        \begin{claim}
            $\langle H^{\alpha, \beta}_\gamma: \alpha, \beta<\mathfrak{d}, \gamma<\non(\mathcal{M})\rangle$ is a basis of $\mathcal{H}$.
        \end{claim}
        Take any $H(x, \bar{J})$. Of course we can find $\alpha$ such that $H(x,\bar{J})\subseteq H(x, \bar{I}^\alpha)$. There exists $f^\alpha_\gamma$ such that for infinitely many $n$ we have $x\upharpoonright_{I_n^\alpha}=f_\gamma^\alpha\upharpoonright_{ I_n^\alpha}$.  Take interval partition $\bar{R}$ such that for almost all $n$ there exists $k$ such that $I^\alpha_k\subseteq R_n$ and $x\upharpoonright_{ I_k^\alpha} =f^\alpha_\gamma\upharpoonright_{ I_k^\alpha}.$ Take $I^\beta$ such that $R\sqsubseteq^* I^\beta$. Of course $H(x,\bar{J})\subseteq H(f^\alpha_\gamma
        , I^\beta)$
\end{proof}

We will finish this section with calculation of some transitive coefficients. 
\begin{df}
    For a $\sigma$-ideal $\mathcal{I}$ on a Polish group $(X,+)$ we define \emph{transitive covering} and \emph{transitive cofinality} of $\mathcal{I}$ as:
    \begin{itemize}
        \item[--] $\cov_t(\mathcal{I})=\min\{|\mathcal{A}|: \exists Y \in \mathcal{I} \;\; Y+\mathcal{A}=X\}$,
        \item[--] $\cof_t(\mathcal{I})=\min\{|\mathcal{A}|: \mathcal{A}\subseteq \mathcal{I} \; \wedge \;\forall B\in\mathcal{I}\;\;\exists A\in \mathcal{A} \; \exists  x\in X\;\;  B\subseteq A+x\}$.
    \end{itemize}
\end{df}
Of course for any $\sigma$-ideal $\mathcal{I}$ we have $\cov_{t}(\mathcal{I})\geq \cov(\mathcal{I})$ $\cov_t(\mathcal{I})=\non(\mathcal{I}^*)$ (see the definition of the star operation in chapter $4$). Such coefficients were investigated for measure and category by Pawlikowski in \cite{JanuszTransitiveBasisNM}. In the context of Polish groups, the transitive covering number was investigated by Stepr\=ans and Miller in \cite{MILLER200652}. In particular, they showed that $\cov_t(\mathcal{M})$ on $\mathbb{Z}^\omega$ is equal to $\cov(\mathcal{M})$. We will prove the same with the ideal $\mathcal{M}_-$.
\begin{proposition}\label{covt}
    $\cov_t(\mathcal{M}_-)=\cov(\mathcal{M})$
\end{proposition}
\begin{proof}
        As $\cov(\Mminus)=\cov(\Meager)$ only needs to show the inequality  $\cov_t(\mathcal{M}_-)\leq \cov(\mathcal{M})$. Take $F=\{x_\alpha: \alpha< \cov(\mathcal{M})\}$ and $N\subseteq M(0,\bar{I})$ for some $\bar{I}\in \mathbb{IP}.$ So $N+F=\bigcup_{\alpha<\cov(\mathcal{M})} M (x_\alpha, \bar{I})$. Assume that there exists $y\notin N+F$. So, for every$\alpha$, there exist infinitely many $n$ such that $ y\upharpoonright_{ I_n} =x_\alpha \upharpoonright_{ I_n}$. Thus, for every $\alpha$, there exist infinitely many $n$ such that $y(n)=x_\alpha(n)$. This leads to a contradiction because $F$ is of size $\cov(\mathcal{M})$.
\end{proof}
The transitive cofinalities of the ideals $\Mminus$ and $\mathcal{H}$ behave as follows.
\begin{thm}
The following hold:
\begin{itemize}
    \item[--] $\cof_t(\mathcal{M}_-)=\mathfrak{d}$,
    \item[--] $\cof_t(\mathcal{H})=\mathfrak{d}$.
\end{itemize}
\end{thm}
\begin{proof}
    To show that $\mathfrak{d}\geq \cof_t(\mathcal{M}_-)$, assume that we are given $\{M(0,\bar{I}^\alpha): \alpha < \mathfrak{d}\}$. Take any $M(y,\bar{J})$. We can find $\bar{I}^\alpha$ such that $\bar{J}\sqsubseteq ^* \bar{I}^\alpha$. Of course $M(y, \bar{J})\subseteq M(y,\bar{I}^\alpha)=M(0, \bar{I}^\alpha)+y$.

    To show that $\mathfrak{d\leq \cof_t(\mathcal{M}_-)}$ assume that we are given a family $\{M(0,\bar{I}^\alpha): \alpha < \cof _t (\mathcal{M}_-)\}$. For any $B\subseteq M(y,\bar{J})\in \mathcal{M}_-$ we can find $\bar{I}^\alpha$ and $x\in \omega^\omega$ such that $B\subseteq M(0,\bar{I}^\alpha)+x$. Take any $\bar{J}\in \mathbb{IP}$ and $M(y, \bar{J}).$ We can find $M(0,\bar{I}^\alpha)+y \supseteq M(y, \bar{J}). $ So we get $\bar{J}\sqsubseteq^* \bar{I}^\alpha$.

     To show that $\mathfrak{d}\geq \cof_t(\mathcal{H})$, take the family $\{H(0,\bar{I}^\alpha): \alpha < \mathfrak{d}\}$. Take any $H(f,\bar{I}).$ We can find $\bar{I}^\alpha$ such that $\bar{J}\sqsubseteq ^* \bar{I}^\alpha$. Of course $H(f,\bar{J})\subseteq H(f,\bar{I}^\alpha)=H(0,\bar{I}^\alpha)+f$

    Finally, to show that $\mathfrak{d\leq \cof_t(\mathcal{H})}$ take $\{H(0,\bar{I}^\alpha): \alpha < \cof_t(\mathcal{H})\}$. For any $B\subseteq H(y,\bar{J})\in \mathcal{H}$ we can find $\bar{I}^\alpha$ and $x\in \omega^\omega$ such that $B\subseteq H(0,\bar{I}^\alpha)+x$. Take any $\bar{J}\in \mathbb{IP}$ and $H(y,\bar{J}).$ We can find $H(0,\bar{I}^\alpha)+y \supseteq H_{ \bar{J},y}. $ So we get $\bar{J}\sqsubseteq^* \bar{I}^\alpha.$ 
\end{proof}

\section{Star operation}
If $(X,+)$ is (an abelian) group, then
for $A, B \in \mathcal{P}(X)$, 
we will write 
$$A+B= \{a+b: a \in A, b \in B\}.$$
The notion of star operation was first explicitly introduced by Seredyński in \cite{Seredyski1989SomeOR} and was implicitly used by Galvin, Mycielski, and Solovay before in \cite{galvin1973strong} to characterize strong measure zero sets.
\begin{df}
For a family $\mathcal{F}\subseteq \mathcal{P}(X)$, where $X$ is a Polish group let:
$$\mathcal{F}^{*}=\{A\subseteq X:\ \forall F\in \mathcal{F} \;\; A+F \neq X \}.$$
\end{df}
Strong measure zero sets were originally defined on the real line by Borel, but we can translate the definition   from $\mathbb{R}$ to $2^\omega$. Recall that a set $X \subseteq 2^\omega$ has strong zero measure if, for every sequence $(k_n)_{n<\omega}$ of natural numbers, there exists a sequence $(\sigma_n)_{n<\omega}$ such that  $X\subseteq \bigcup_n [\sigma_n]$ and $|\sigma_n|=k_n$ for every $n<\omega$. The following relation between strong measure zero sets and meager sets on the real line was first shown in \cite{galvin1973strong} by Galvin, Mycielski, and Solovay. 

\begin{thm}(Galvin, Mycielski, Solovay; \cite{galvin1973strong})
    $\mathcal{SMZ}=\Meager^{*}.$
\end{thm}
The star operation was used to define the notion of strongly meager sets (see \cite{BBB}) and was also later studied  by Pawlikowski and Sabok who showed in \cite{Pawlikowski2008-PAWTS} that $Ctbl = Ctbl^{**}$. Moreover, Solecki in \cite{solecki} constructed a $\sigma$-ideal $\mathcal{I}$ satisfying $\mathcal{I}^* = Ctbl$.

We begin with some basic facts about $^*$operation.

\begin{fact}\label{gwiazdka}
    For any $\mathcal{F,G} \subseteq \mathcal{P}(X)$ we have:
    \begin{itemize}
    \item[--] $\mathcal{F}\subseteq\mathcal{F}^{**}$,
    \item[--] $\mathcal{F}^*=\mathcal{F}^{***}$,
    \item[--] $\mathcal{F}^{*}$ is closed under taking subsets and translation invariant,
    \item[--] if $\mathcal{G}\subseteq\mathcal{F}^{*}$, then $ \mathcal{F}\subseteq\mathcal{G}^{*}$,
    \item[--] if $\mathcal{G}\subseteq\mathcal{F}$, then $ \mathcal{F}^*\subseteq\mathcal{G}^{*}$,
    \item[--] $\mathcal{F}=\mathcal{F}^{**}$ if and only if $ \exists \mathcal{A}\; \;\;\mathcal{F}=\mathcal{A}^{*}.$ 
\end{itemize}
\end{fact}
From now on we will be interested only in the $\sigma$-ideals on the Baire space $\mathbb{Z}^{\omega}$ seen as a Polish group with coordinatewise addition (called Baer-Specker group).

Now, we will rewrite definition and characterizations of  strong measure zero sets into the Baire space $\mathbb{Z}^{\omega}$. 
\begin{df}
    A set $A \subseteq \mathbb{Z}^\omega$ has strong zero measure $(A\in \mathcal{SMZ})$ if for every sequence $(k_n)_{n<\omega}$ of natural numbers there exists a sequence $(\sigma_n)_{n<\omega}$ such that  $A\subseteq \bigcup [\sigma_n]$ and $|\sigma_n|=k_n$  for every $n<\omega$.
\end{df}
In \cite{Wohofsky}, Wohofsky showed that the Galvin-Mycielski-Solovay Theorem doesn't hold in the Baire space $\mathbb{Z}^\omega$. He showed that $\mathcal{M}^*$ is contained in strong measure zero sets, and that under the assumption of the Continuum Hypothesis it is possible to construct a  set $X\in \mathcal{SMZ}\backslash \mathcal{M}^*$.

The same argument as the one for Cantor space can be used to prove that the above definition is equivalent with the following one.
\begin{proposition}
    A set $X\subseteq \mathbb{Z}^\omega$ has strong measure zero if and only if for every increasing sequence $k_n\in \omega^\omega$ there exists $(\sigma_n)_{n<\omega}$ such that $\sigma_n\in \mathbb{Z}^{k_n}$ and:
     $$X\subseteq\{x:\exists ^\infty n \;\; \sigma_n\subseteq x\}.$$
 \end{proposition}

As it turns out, the characterization of strong measure zero sets provided before does not hold in the Baire space $\mathbb{Z}^{\omega}$. This motivates the following definition.

\begin{df}
For $X\subseteq\mathbb{Z}^{\omega}$ we define:
  $X\in \mathcal{SMZ}^+$  if and only if for every interval partition $\bar{I}\in \mathbb{IP}$ there exists $z\in\mathbb{Z}^\omega$ such that:
 $$\forall x\in X\;\; \exists^\infty n\;\; x\upharpoonright_{I_n}=z\upharpoonright_ {I_n}.$$

\end{df} 
The motivation for the family $\mathcal{SMZ}^{+}$ is the following characterization of strongly measure zero sets in $2^{\omega}$:

It is not difficult to see that $\mathcal{SMZ}\subset\mathcal{SMZ}^{+}$.

 The following is a version of the Galvin-Mycielski-Solovay theorem
for the Baire space in terms of the ideals introduced above.

\begin{thm}
$\mathcal{M}_-^*=\mathcal{SMZ}^+.$
\end{thm}

\begin{proof}
   To show $\mathcal{M}^*_-\subseteq\mathcal{SMZ}^+$ assume $X+H\neq \mathbb{Z}^\omega $ for any $H\in \mathcal{M_-}$. Take an interval partition $\bar{I}\in \mathbb{IP}$ and take $H\subseteq M(y,\bar{I})$, by translating $H$ we may assume that $y$ is constant with value $0$ on every coordinate. Since $X+H\neq \mathbb{Z}^\omega$ there exists $z\notin X+H\subseteq\bigcup_{x\in X} M(x,\bar{I}).$ So for every $x\in X$ there exists infinitely many $n $ such that $z\upharpoonright_ {I_n} = x\upharpoonright_{ I_n}$. It follows that $X\in\mathcal{SMZ}^+$.\\
    To show $\mathcal{SMZ}^+\subseteq\mathcal{M}_-^*$
    take any $X\in \mathcal{SMZ}^+$ and any $H\in \mathcal{M}_-$, where $H\subseteq M(y,\bar{I})$ for some $y\in \mathbb{Z}^\omega$ and $\bar{I}\in \mathbb{IP}.$ Since $X\in \mathcal{SMZ}^+$ we have $X\subseteq\{x: \exists^{\infty} n \;\; x\upharpoonright_{ I_n} =z\upharpoonright_ {I_n}  \}$. We have $X+H\subseteq \bigcup_{x\in X} M(y+x,\bar{I})$. We have to show that there exists $a\notin X+H$. Let $a=y+z$.
Assume that $y+z\in X+H$, then for almost all $n$ and some $x\in X$ we have
$(y+x) \upharpoonright_ {I_n} \neq (y+z)\upharpoonright_ {I_n}$ , so $x\upharpoonright_ {I_n}\neq z\upharpoonright_ {I_n}$. 
We get a contradiction, because $X\in \mathcal{SMZ}^+$, so for infinitely many $n$ we have $x\upharpoonright_{ I_n} =z \upharpoonright_{ I_n}$.
\end{proof}
\begin{proposition}
  $X\in \mathcal{H}^{*}$  if and only if for every interval partition $\bar{I}\in \mathbb{IP}$ there exists $z\in\mathbb{Z}^\omega$ such that:
 $$\forall x\in X\;\; \exists^\infty n\;\; \forall m\in I_n \;\; x(m)\neq z(m).$$
\end{proposition}
\begin{proof} 
$(\subseteq)$ Take any $X $ satisfying the characterization from theorem and $H({f, \bar{I}})\in \mathcal{H}$. We want to find $z \notin X + H({f, \bar{I}})$. 
Since $X \in \mathcal{H}^*$, we can find $g$ such that there exist infinitely many $n$ such that for all
$m \in I_n $, we have $ x(m) \neq g(m)$. Let $z(m) = f(m) + g(m)$.

\noindent
$(\supseteq)$ Assume $X $ does not satisfy the characterization from the theorem. We will show that there exist $\bar{I}\in \mathbb{IP}$ and $g\in\mathbb{Z}^\omega$ such that $X + H({g,\bar{I}}) = \mathbb{Z}^\omega$. Since $X $ does not satisfy the characterization, the negation of the definition holds:
\[
\exists \bar{I} \in \mathbb{IP} \;\; \forall f \in \mathbb{Z}^\omega \;\; \exists x \in X \;\; \forall^\infty  n \;\; \exists m \in I_n \;\; x(m) = f(m).
\]
Let $\bar{I}$ be such partition. Take any $g \in \mathbb{Z}^\omega$ and $a \in \mathbb{Z}^\omega$. We want to find $x \in X$ and $y \in H({g,\bar{I}})$ such that $a = x + y$. We know that $a=f+g$ for some $f$. Take $x_f\in X$ such that for almost all $n$ there exists 
$ m\in I_n $ such that $ x(m)=f(m).$ Let $x=x_f$. Now we will construct $y\in H({g,\bar{I}})$ such that $a=x_f+y$. 
\begin{enumerate}
    \item If $f(n)=x_f(n)$, let $y(n)=g(n)$;
    \item else let $y(n)=a(n)-x_f(n)$.
\end{enumerate}
 Of course, $y\in H({g,\bar{I}})$ because $y(n)=g(n)$ if and only if $x_f(n)=f(n)$.
\end{proof}

\begin{proposition}
    $\mathcal{SMZ}^+$ and $\mathcal{H}^*$ form a $\sigma$-ideal on the Baire space.
\end{proposition}
\begin{proof}
    In case of $\mathcal{SMZ}^{+}$ assume that $(X_l: l<\omega)$ where each $X_l \in \mathcal{SMZ}^+.$
    Take any interval partition $\bar{I}.$ We will construct a bigger partition $\bar{J}.$ Let:
    $$ J_n=I_{\frac{n\cdot(n-1)}{2}+n}\cup \cdots \cup I_{\frac{n\cdot(n-1)}{2}+2n}.$$ 
    
    Since every $X_l$ is in $\mathcal{SMZ}^+$ there exists $z_l$ for every $X_l$ such that: $$\forall x\in X_l \;\; \exists^\infty n  \;\; x\upharpoonright_{ J_n}= z_l\upharpoonright_{ J_n}.$$ We want to construct $z$ such that for every $l$, for all $x\in X_l $  there exists infinitely many $n$  such that $z\upharpoonright_{ J_n}= x\upharpoonright_ {J_n}$.
    
    Let $z\upharpoonright_{ I_0} =z_1\upharpoonright_{ I_0}$.
    Next, if $k=\frac{n(n-1)}{2}+n\) for some \(n\) and $0<b<k+1$ define 
$z\upharpoonright_ {I_k}=z_1\upharpoonright_{ I_k}$ and  
  $z\upharpoonright_ {I_{k+b}}=z_{b+1}\upharpoonright_ {I_{k+b}}$.\\

  In case of $\mathcal{H}^*$ Assume that $\{X_l: l<\omega\}\subseteq\mathcal{H}^*$ and let $\bar{I}$ be an interval partition. Let $\bar{J}=(J_{n})_{n\in\omega}$ be such interval partition that:
   $$J_n=I_{\frac{n\cdot(n-1)}{2}+n}\cup \cdots  \cup  I_{\frac{n\cdot(n-1)}{2}+2n}.$$
    Since every $X_l$ is in $\mathcal{H}^*$ there exists $z_l$ for every $X_l$ such that: \[\forall x\in X_l \;\; \exists^\infty n\; \forall m\in J_n\; x(m)\neq z_l(m).\] We will construct $z$ such that for every $l$, for every  $x\in X_l$ there exists infinitely many $n$ such that for every $m\in J_n \;\; z(m)\neq x(m)$. 
    
    Let $z\upharpoonright_{ I_0} =z_1\upharpoonright_ {I_0}$.
    Next, if $k=\frac{n(n-1)}{2}+n\) for some \(n\) and $0<b<k+1$ define 
$z\upharpoonright_ {I_k}=z_1\upharpoonright_{ I_k}$ and  
  $z\upharpoonright_ {I_{k+b}}=z_{b+1}\upharpoonright_ {I_{k+b}}$.

\end{proof}

As $\mathcal{H}\subseteq\Mminus$, from Proposition \ref{gwiazdka} it follows that $\mathcal{SMZ}^+\subseteq\mathcal{H}^*$. Now, we will show some relation between the ideals defined in Chapter $2$ with respect to the $^*$operation. 
\begin{proposition}
    $X\in \mathcal{ED}^*$ if and only if $X\subseteq \bigcup I_f$ where $f\in \mathbb{Z}$ and $$I_f=\{x\in \mathbb{Z}^\omega : \exists^\infty n\;\;\;\; x(n)=f(n)\}$$

\end{proposition}
\begin{proof}
During the proof it will be enough to work with basis of $\mathcal{ED}$.

Take any $X\in \mathcal{ED}^*$. For every $E_f \in \mathcal{ED}$ there exists $z\notin \bigcup_{x\in X} E_f+x$. So for every $x\in X$ we have $x\notin -E_f+z=E_g$ for some $g\in \mathbb {Z}^\omega$. So $x\in\{y:\exists^\infty n\; y(n)=g(n)\}$ and so $X$ satisfy the condition from the theorem.
        
 Take any $I_g$ satisfying condition from the theorem, and any $E_f\in \mathcal{ED}$. Then, notice that $z=f+g$ is such that $z\notin E_f+I_g$.

\end{proof}

\begin{proposition}
      $\mathcal{ED}^{**}=\mathcal{ED}$.
\end{proposition}

\begin{proof}
We will start by showing $\mathcal{ED}^{**}\subseteq\mathcal{ED}$. Take any $X\in \mathcal{ED}^{**}$. For every $I_f \in \mathcal{ED}^{*}$ there exists $z\notin \bigcup_{x\in X} I_f+x$. So for every $x\in X$ we have $x\notin -I_f+z=I_g$ for some $g\in \mathbb {Z}^\omega$. So $x\in\{y:\forall^\infty n \;\; g(n)\neq y(n)\}$ and so $X\subseteq E_g$.

Next we will show $\mathcal{ED}\subseteq\mathcal{ED}^{**}$. Take any $E_f\in \mathcal{ED}$ and any $I_g\in\mathcal{ED}^{*}$. Then, notice that $z=f+g$ is such that $z\notin E_f+I_g$.
    
\end{proof}

\begin{thm}\label{um*}
    $X\in \mathcal{UM}^*$ if and only if for every $\bar{I}\in \mathbb{IP}$ and every $\psi: \bigcup_n \omega^{min(I_n)}\rightarrow \bigcup _n \omega^{I_n}$ we have:
$$\forall x\in X\;\;\;\; \exists^\infty n \;\;\; x\upharpoonright _{I_n}= \psi (x\upharpoonright _{min(I_n)}). $$
\end{thm}
\begin{proof}
    $(\supseteq)$ take any  $Y\in \mathcal{UM}$ and let $X$ satisfy the combinatorial characterization from the theorem. We will find $z\notin X+Y.$ We know from Proposition \ref{umlema} that $Y\subseteq U(\psi, \bar{I})$ for some $\bar{I}\in \mathbb{IP}$ and  $\psi: \bigcup_n \omega^{min(I_n)}\rightarrow \bigcup _n \omega^{I_n}$. We know that for all $x\in X$ there exists infinitely many $n$ such that $x\upharpoonright _{I_n}= \psi (x\upharpoonright _{min(I_n)})$. We will construct $z$ as follows:

    \begin{enumerate}
 \item $z\upharpoonright_{I_0}=\psi(\emptyset)+\psi(\emptyset)$
 \item $z\upharpoonright_{I_n}= 2\psi (z\upharpoonright_{min(I_n)})$ 
    \end{enumerate}

$(\subseteq)$ Assume that there exists $X\in \mathcal{UM}^*$ such that $X$ does not satisfy the characterization from the theorem. So there exist $\bar{I}\in \mathbb{IP}$, $\psi: \bigcup_n \omega^{min(I_n)}\rightarrow \bigcup _n \omega^{I_n}$, and $x\in X$ and $n_0$ such that for  all $n>n_0$, we have  $x\upharpoonright _{I_n}\neq \psi (x\upharpoonright _{min(I_n)})$. Take any $z\in \mathbb{Z}^\omega$. 
Let $$\phi(\sigma)= -\psi(-\sigma+z\upharpoonright_{min(I_n)})+z\upharpoonright_{I_n}.$$
Notice that $z-x\in U(\phi,\bar{I})$. Indeed  for $n>n_0$ we have

$(z-x)\upharpoonright _{I_n}= z\upharpoonright_{I_n}-x\upharpoonright _{I_n}\neq -\psi((x-z)\upharpoonright_{I_n} +z\upharpoonright_{min(I_n)})+z\upharpoonright_{I_n}=\phi((z-x)\upharpoonright_{I_n}).$

\end{proof}

More general Theorem was showed in \cite{umeager}, but we present an easier proof.

\begin{thm}
    $\mathcal{SMZ}\subseteq \mathcal{UM}^*$.
    
\end{thm}

\begin{proof}
    Take any $X\in\mathcal{SMZ}$ and $H\in \mathcal{UM}$ by Theorem \ref{UM} there are $\bar{I}\in \mathbb{IP}$ and  $\psi:\bigcup_n\mathbb{Z}^{\min(I_{n})}\rightarrow \bigcup_n\mathbb{Z}^{I_n} $ such that $H\subseteq U(\psi,\bar{I})$.  Let $(\sigma_n)_{n<\omega}$ be such that $|\sigma_n|=\max(I_n)$ and  $X\subseteq\{x:\exists ^\infty n \;\; \sigma_n\subseteq x\}.$

    We will construct $\tau\notin X+H$. Let:
    \begin{enumerate}

        \item  $\tau_0=\psi(\emptyset)+\sigma_0$
        \item $\tau_n= \tau_{n-1}^\frown (\psi (\tau_{n-1} -\sigma_n \upharpoonright_{ (I_0,\cup \cdots \cup I_{n-1})})+\sigma_n\upharpoonright_{ I_n}) $
    \end{enumerate}

Let $\tau=\bigcup_{n} \tau_n$. Of course $\tau \notin X+H$, so $\mathcal{SMZ}\subseteq\mathcal{UM}^*$
\end{proof}

\begin{thm}
    $X\in \mathcal{FM}^*$  if and only if for every $\psi: \omega^{<\omega}\rightarrow \omega$ we have:
$$\forall x\in X\;\;\;\; \exists^\infty n \;\;\; x(n)= \psi (x\upharpoonright _n). $$ 
\end{thm}

\begin{proof}
    Proof is very similar to the proof of Theorem \ref{um*}. 
\end{proof}

 Since we have $\mathcal{H}\subseteq \mathcal{ED}^*$ from Proposition \ref{gwiazdka} we get $\mathcal{ED}\subseteq\mathcal{H}^*$. It also follows that $\mathcal{ED}=\mathcal{ED^{**}}$ and $\mathcal{SMZ}^+\subseteq\mathcal{ED}^*$. The following diagram presents the relationships between the ideals, including $\mathcal{SMZ}$ and $\mathcal{SMZ}^+$. 

 We don't know any characterization of $\mathcal{P}^*$ yet. More on this topic can be found in \cite{cccccc}.
 
\begin{center}
\begin{tikzpicture}[node distance=1.5cm, auto]
    \node (M) at (0, 4.5) {$\mathcal{M}$};
    \node (MM) at (0, 3) {$\mathcal{UM}$};
    \node (Mminus) at (0, 1.5) {$\mathcal{M}_-$};
    \node (IE) at (-2, 1.5) {$\mathcal{ED}^*$};
    \node (G) at (2, 1.5) {$\mathcal{H}^*$};
    \node (H) at (-2, -0.5) {$\mathcal{H}$};

    \node (F) at (-3, 0.5) {$\mathcal{FM}^*$};
    
    \node (ED) at (2, -0.5) {$\mathcal{ED}$};
    \node (SMZ) at (0,-0.5)
    {$\mathcal{SMZ}^+$};
        \node (SMZZ) at (0,-3.5)
 {$\mathcal{SMZ}$};

       \node (PS) at (-4,-0.45)
   {$\mathcal{P}^*$};
             \node (U) at (0,-2)     
    {$\mathcal{UM}^*$};
    
        \node (P) at (3.8, 1.45) 
    {$\mathcal{P}$};
            \node (FM) at (3, 0.5) 
    {$\mathcal{FM}$};

            \draw[->] (ED) -- (FM);
              \draw[->] (FM) -- (P);
            \draw[->] (P) -- (MM);
    \draw[->] (H) -- (IE);
    \draw[->] (H) -- (Mminus);
    \draw[->] (ED) -- (Mminus);
    \draw[->] (ED) -- (G);
    \draw[->] (Mminus) -- (MM);
    \draw[->] (MM) -- (M);
    \draw[->]  (SMZ)-- (IE);
    \draw[->]   (SMZ) -- (G);

        \draw[->]   (SMZZ) -- (U);
            \draw[->]   (U) -- (SMZ);
       \draw[->]   (U) -- (PS);    
           \draw[->]   (F) -- (IE);
            \draw[->]   (PS) -- (F);          
\end{tikzpicture}
\end{center}

\begin{remark}
    One may consider the definitions of the $\sigma$-ideals $\mathcal{H}$, $\mathcal{M}_-$ and $\mathcal{UM}$ in Cantor space instead of Baire space but it is not difficult to see that then all of these ideals are equal the $\sigma$-ideal of meager subsets of $\Cantor$.\\
    If one also consider the definitions of $\mathcal{H}^*$ and $\mathcal{SMZ}^+$ on the Cantor space $\Cantor$ instead of $\Baire$, then these $\sigma$-ideals are equal to the standard strong measure zero sets in $\Cantor$.
\end{remark}

Now we will use Luzin set to show that some inclusions does not hold in $ZFC.$

 Let $X$ be a Polish space and $\mathcal{I}$ be an ideal. Then we call a set $L \subseteq X$ an $\mathcal{I}$-Luzin set if $|L \cap A|<|L|$ for all $A \in \mathcal{I}$.
It is known that $\mathcal{M}$-Luzin set has strong measure zero on $2^{\omega}$, but is not meager. In follows that every $\mathcal{M}$-Luzin is in $\mathcal{M}^*$. So assuming the existence of a Luzin set, we have $\mathcal{H}^*\nsubseteq\mathcal{M}_-$, $\mathcal{ED}^*\nsubseteq \mathcal{M}_-$, $\mathcal{SMZ}\nsubseteq \mathcal{ED}$ and $\mathcal{SMZ}\nsubseteq\mathcal{H}$. We do not know if every $\mathcal{M}$-Luzin is $\mathcal{M}_-^*$?

Next, we will use the  $^*$operation  to show that some of the inclusions the diagram are proper. 
\begin{proposition}
        There exists $X\in \mathcal{H}$ such that $X\notin \mathcal{SMZ}^+.$
\end{proposition}
\begin{proof}

    Take $\bar{I}$ interval partition into intervals of length two. We will show that $H({1,\bar{I}})+M(0,\bar{I})=\mathbb{Z}^\omega$. Take any $z\in \mathbb{Z}^\omega$, and construct $y\in M(0,\bar{I})$ and $x\in H({1,\bar{I}})$ such that:
    \begin{enumerate}

        \item If $n$ is odd, let $x(n)=1$ and $y(n)=z(n)-1$,
        \item if $n$ is even, let $x(n)=z(n)-1$ and $y(n)=1.$
         
    \end{enumerate}

    Of course $y\in M(0, \bar{I})$, because in every interval we can find $m$ such that $y(m)=1$. Similarly $x\in H({1,\bar{I}})$ as in every interval we can find $m$ such that $x(m)=1$.
\end{proof}

\begin{cor}
    There exists $X\in \mathcal{ED}^*$ such that $X\notin \mathcal{SMZ}^+.$
\end{cor}
\begin{thm}
    There exists $Y\in \mathcal{ED}$ such that $Y\notin \mathcal{SMZ}^+.$
\end{thm}

\begin{proof}
    Let $Y=\{x:\forall^\infty n \;\; x(n)\neq 0\}\in \mathcal{ED}.$ We will show that $Y\notin \mathcal{M}_-^*=\mathcal{SMZ}^+.$ Take $M(0,\bar{I})$, where $\bar{I}$ is partition into singletons. Take any $a\in \mathbb{Z}^\omega.$ We have to find $f\in Y$ and $g\in M(0,\bar{I})$ such that $f+g=h$.
    \begin{enumerate}
        \item If $h(n)\neq 1$ let $f(n)=h(n)-1$ and $g(n)=1$,
        \item  If $h(n)=1$ let $f(n)=2$ and $g(n)=-1$.
    \end{enumerate}
    Of course $g\in M(0,\bar{I})$, because $g(n)\neq 0$ for every $n$. Also $f\in Y$, because $f(n)\neq 0$ for every $n$.
\end{proof}

\begin{cor}
    There exists $Y\in \mathcal{H}^*$ such that $Y\notin\mathcal{SMZ}^+$.
\end{cor}

Next, we will need a Theorem from \cite{cccccc} to show that, under additional assumptions, we can separate some of these ideals.

\begin{thm}(Perkowska, Żeberski; \cite{cccccc}) \label{bbbbb}
    If $\mathcal{J} \subseteq \mathcal{P}(\mathbb{Z}^{\omega})$ is a translation  and reflection invariant proper $\sigma$-ideal such that  $\cof(\mathcal{J})=\add(\mathcal{J})= \mathfrak{c}$ we have that for every $A \notin \mathcal{J}$:
$$(\mathcal{J} \cup \{A\})^* \neq \mathcal{J}^*.$$
\end{thm}
As a corollary we get the following.
\begin{cor}
    Assuming $\add(\mathcal{M})=\mathfrak{c}$ there exists $X\in \mathcal{H}^*$ such that  $X\notin\mathcal{ED}$.
\end{cor}

\begin{proof}
    
By Theorem $\add(\mathcal{M})=\add(\mathcal{H})=\mathfrak{c}$ from Theorem \ref{bbbbb} we get $\mathcal{ED}\neq \mathcal{H}^*.$
\end{proof}

As $\mathcal{SMZ}\subseteq\mathcal{UM}^*$, by Theorem \ref{bbbbb} we get the following.
\begin{cor}
 $\add(\mathcal{M})=\mathfrak{c}$ implies that there is $X\in \mathcal{SMZ}^+\setminus \mathcal{SMZ}$.
\end{cor}

It follows from the results proven so far that the cardinal invarians $\cov_t(\mathcal{M}_-)$, $\non(\mathcal{SMZ}^+)$, $\non(\mathcal{SMZ})$, $ \non(\mathcal{M}^*)$, $\cov_t(\mathcal{M})$ are all equal $\cov(\mathcal{M})$: from Proposition \ref{covt} we know that $\cov(\mathcal{M})=\cov_t(\mathcal{M}_-)$. Since $\mathcal{M}_-^*=\mathcal{SMZ}^+$ we get the next equality. Since $\mathcal{M}^*\subseteq\mathcal{SMZ}\subseteq \mathcal{SMZ}^+$ we have $\non(\mathcal{M}^*)\leq \non (\mathcal{SMZ})\leq \non(\mathcal{SMZ}^+)$. Of course $\non(\mathcal{M}^*)=\cov_t(\mathcal{M})$. Also we have $\cov(\mathcal{M})= \cov_t(\mathcal{M})$. So we get that all the above cardinal invariants are equal to $\cov(\mathcal{M}).$

We finish this section with a two result regarding translations of $\sigma$-compact subsets of $\mathbb{Z}^{\omega}$.
\begin{proposition}
    $K_\sigma^* $ does not form an ideal.
\end{proposition}
\begin{proof}
    Let $\mathcal{U}$ be an ultrafilter on $\omega$.
    Let $A=\{x\in\mathbb{Z}^\omega: \{n:x(n)\geq0\}\in \mathcal{U}\}$.
    Let $B=\mathbb{Z}^\omega\backslash A=\{x\in\mathbb{Z}^\omega: \{n:x(n)<0\}\in \mathcal{U}\}$
    Note that  $A,B\in K_\sigma^*$ and $A\cup B=\mathbb{Z}^\omega$ it follows that $K_\sigma^*$ is not an ideal.
\end{proof}
\begin{proposition}
    There exists $M\in\mathcal{M}_-$ and $K\in K_\sigma $ such that $M+K=\mathbb{Z}^\omega.$
\end{proposition}
\begin{proof}
We will show that this holds for the sets $M=M(0,\bar{I})$, where $\bar{I}$ is a partition into singletons and $K=\{x\in \mathbb{Z}^\omega:\forall^\infty_n \;\;|x(n)|<2\}$. Take any $z\in \mathbb{Z}^\omega$. We construct $x\in M$ and $y\in K$ such that $x+y=z$ as follows:
\begin{enumerate}
    \item If $\forall^\infty_n z(n)\neq0$, then $z\in M$, so $z=z+0$,
    \item If $\exists^\infty_n \ z(n)=0$, take $x,y$ as follows:
    \begin{itemize}
        \item [--] if $z(n)\neq0$ let $x(n)=z(n)$ and $y(n)=0.$
        \item [--]  if $z(n)=0$ let $x(n)=1$ and $y(n)=-1.$
    \end{itemize}
\end{enumerate}
\end{proof}

\section{Tree dichotomies}
Tree dichotomies play an important role in descriptive set theory as well as 'idealized forcing' theory, extensively developed by Zapletal in \cite{zapletal} and \cite{MR2849260}. Such dichotomies assert that given an $\sigma$-ideal $\mathcal{I}$ on the reals, every Borel (or analytic) $\mathcal{I}$-positive set contains the set of branches of a tree of a certain type. The two most well-known such dichotomies are the ones stating that every uncountable analytic set contains an uncountable perfect set (which is the set of branches of a perfect tree) or that every analytic subset of $\omega^{\omega}$ that is not coverable by a countable union of compact sets contains the set of branches of a superperfect (Miller) tree. Similar dichotomies were established in \cite{Millller}, \cite{KHOMSKII}, \cite{MR2039350}, \cite{MR1112500} or \cite{MR2849260}. In particular, the following dichotomies were established for the $\sigma$-ideals $\mathcal{FM}$ and $\mathcal{ED}$.
\begin{df}
    For $T\subseteq \omega^{<\omega}$ let:
    \begin{itemize}
        \item[--] $T\in \mathbb{T}_{\mathcal{FM}}$ if $\sigma\in T$ has an extension $\tau\in T$ such that $succ_{T}(\sigma)=\omega$.
        \item[--] $T\in \mathbb{T}_{\mathcal{ED}}$ if for every $\sigma\in T$ there is $N>|\sigma|$ such that the spectrum $\{\tau(N):\tau\in T|_{\sigma}\}$ is equal $\omega$.
    \end{itemize}
\end{df}
The two tree types defined above stands in the perfect-set style dichotomy with the $\sigma$-ideals $\mathcal{FM}$ and $\mathcal{ED}$.
\begin{thm}(Khomskii and Laguzzi; \cite{KHOMSKII})
    For every analytic set $A\subseteq\Baire$, either $A\in \mathcal{FM}$ or $[T]\subseteq A$ for some $T\in \mathbb{T}_{\mathcal{FM}}$. In particular, the assignment 
    \begin{center}
        $\mathbb{T}_{\mathcal{FM}}\ni T\longmapsto [T]\in Bor \setminus{\mathcal{FM}}$
    \end{center} is a dense embedding, and the forcing notions $(\mathbb{T}_{\mathcal{FM}}, \subseteq)$ and $(Bor /{\mathcal{FM}}, \subseteq)$ are forcing equivalent.
\end{thm}
\begin{thm}(Spinas \cite{SpinasPerfSetTheorems})
    For every analytic set $A\subseteq\Baire$, either $A\in \mathcal{ED}$ or $[T]\subseteq A$ for some $T\in \mathbb{T}_{\mathcal{ED}}$. In particular, the assignment 
\begin{center}
        $\mathbb{T}_{\mathcal{ED}}\ni T\longmapsto [T]\in Bor \setminus{\mathcal{ED}}$
\end{center} 
is a dense embedding, and the forcing notions $(\mathbb{T}_{\mathcal{ED}}, \subseteq)$ and $(Bor /{\mathcal{ED}}, \subseteq)$ are forcing equivalent.
\end{thm}
In this section, we will investigate such dichotomies for the ideals $\mathcal{M}_-$, $\mathcal{H}$, and $\mathcal{UM}$. Let us start with the dichotomy for the $\Mminus$.
\begin{df}
    Let $\mathbb{T}_{\mathcal{M}_-} $ be the collection of all trees $T\subseteq\omega^\omega$ such that for every node $\sigma\in T$ there is $N^T_\sigma>|\sigma|$ such that for every interval $I\subseteq \omega \backslash N^T_\sigma$ and every $\rho \in\omega^I$ there is an extension $\sigma' \in T\upharpoonright_{\sigma}$ with $\sigma' \upharpoonright_{ I}=\rho$.
\end{df}
Note that for every $T\in \mathbb{T_{\mathcal{M}_-}}$, we have that $[T]\notin\mathcal{M}_-$. We will show that such trees are canonical positive sets for the $\mathcal{M}_-$ ideal.
\begin{thm}\label{mmm}
    For any analytic $A\subseteq\omega^{\omega}$, either $A\in \mathcal{M}_-$ or there is a tree $T \in \mathbb{T}_{\mathcal{M}_-}$ such that $[T]\subseteq A$. In particular, the assignment 
    \begin{center}
        $\mathbb{T}_{\mathcal{M}_-}\ni T\longmapsto [T]\in Bor \setminus{\mathcal{M}_-}$
    \end{center} is a dense embedding, and the forcing notions $(\mathbb{T}_{\mathcal{M}_-}, \subseteq)$ and $(Bor /{\mathcal{M}_-}, \subseteq)$ are forcing equivalent.
    
\end{thm}

\begin{proof}
     Let $T$ be a tree on $(\omega \times \omega)^{<\omega}$ such that $A=\pi[[T]]$ (here $\pi$ denote the projection on to the first coordinate). For any tree $S$ on $(\omega \times \omega)^{<\omega}$ define the derivative:
     $$S'=\{(\sigma,\tau)\in S:\exists N^T_\sigma>|\sigma|\;\;\forall I\subseteq \omega\backslash N^T_\sigma\;\;\forall \rho \in \omega^I$$$$\exists (\sigma',\tau')\in T\upharpoonright_{(\sigma,\tau)}\;\; \sigma'\upharpoonright_{ I} =\rho\}.$$

Next, we will inductively define:
\begin{itemize}
    \item [--] $T^{0}=T, $
    \item [--] $T^{\alpha+1}=\left(T^{\alpha}\right)^{\prime}$,
    \item [--] $T^{\lambda}=\bigcap_{\alpha<\lambda} T^{\alpha}$ for limit ordinal $\lambda$.
\end{itemize}

Let $\alpha_0<\omega_1$ be the least such that $T^{\alpha_0}=T^{\alpha_0+1}$. Consider the following two cases:

Case $1)$  $T^{\alpha_0}=\emptyset$.

Note that $[\pi (T)]\subseteq \bigcup_{\beta < \alpha_0} [\pi [T^\beta \backslash T^{\beta +1}]].$ It will be enough to show that for every $\beta<\alpha_0$ we have $[\pi[T^\beta\setminus T^{\beta+1}]]\in \mathcal{ M}_-.$ Fix $\beta<\alpha_0$ and $(\sigma, \tau)\in T^\beta \backslash T^{\beta+1}.$ We will show that $[\pi [T^\beta\upharpoonright_{(\sigma, \tau)}]]\in \mathcal{M}_-$.

Since $(\sigma, \tau )\notin T^{\beta+1}$ we get for every $M >|\sigma|$:
$$\exists I\subseteq \omega\backslash M \;\; \exists \rho \in \omega^I \;\; \forall (\sigma',\tau')\in T^\beta \upharpoonright_{(\sigma, \tau)} \;\; \sigma' \upharpoonright_I \neq \rho .\;\; \;\; (*)$$

 We will inductively build  an interval partition $J=(J_{n})_{n\in\omega}$ and $(x_n)_{n<\omega},$  such that $x_n\in \omega^{J_n}$ and $[\pi[T^\beta \upharpoonright_{(\sigma ,\tau )}]]\subseteq M(x,J).$ 
Assume that $(J_k)_{k\leq n+1}$ and $(x_k)_{k\leq n}$ where $x_k \in   \omega^{J_k}$ has been constructed. We use $(*)$ for $M =max(J_n) $. There is an interval $I\subseteq \omega \backslash \max(J_n)$ and $\rho \in \omega^{I}$ such that:
\begin{equation*}
\forall (\sigma',\tau') \in T^\beta \upharpoonright _{(\sigma,\tau)} \;\; \sigma'\upharpoonright_ I \neq \rho \;\;\;\;\;\;\;\;
\;\; \;\; \;(**)
\end{equation*}
Let  $J_{n+1}=[\max(J_n)+1, \max (I)]$ and  let $x\upharpoonright_{  J_{n+1}}$ be such that $x\upharpoonright_ I =\rho$.\\
We claim that $[\pi [T^\beta \upharpoonright _{(\sigma, \tau)}]] \subseteq M(x,\bar{J})$. Take any $y\in [\pi [T^\beta \upharpoonright _{(\sigma, \tau)}]]$. There exists $z\in\omega^{\omega}$ such that $(y,z)\in T^\beta \upharpoonright_{(\sigma,\tau)}$. So by $(**)$   we have $y\upharpoonright_ I \neq \rho$. Since $I\subseteq J_{n+1}$ for every $n\in\omega$, we have that $y\upharpoonright _{J_{n+1}} \neq x_{n+1}\upharpoonright _{J_{n+1}}$ holds for almost all $n\in\omega$. 
Let $\bar{J}=\{J_n\}_{n<\omega}$ and $x=\bigcup_{n<\omega} x_n$.
It follows that  $y\in M(x,\bar{J})$.

Case $2)$  $T^{\alpha_0}\neq \emptyset$.

We will show that $\pi [T^{\alpha_0}] \in \mathbb{T}_{\mathcal{M}_-} $. Let $\sigma \in \pi [T^{\alpha_0}]$ and let $\tau$ be  such that $(\sigma, \tau) \in T^{\alpha_0}.$ As $T^{\alpha_{0}}=T^{\alpha_{0}+1}$ there is $N^T_\sigma>|\sigma|=|\tau|$ such that for every interval $I\subseteq \omega \backslash N^T_\sigma$ and for every $\rho\in\omega ^I$ we can find $(\sigma',\tau')\in T^{\alpha_0}\upharpoonright_{(\sigma,\tau)}$ such that $\sigma'\upharpoonright _I=\rho$. 
\end{proof}

Using a similar argument as the one presented above, we get the tree dichotomy for the $\sigma$-ideal  $\mathcal{H}$.

\begin{df}
    Let $\mathbb{T}_{\mathcal{H}} $ be the collection of all trees $T\subseteq\omega^{<\omega}$ such that for every node $\sigma\in T$ there is $N^T_\sigma>|\sigma|$ such that for every interval $I\subseteq \omega \backslash N^T_\sigma$ and every $\rho \in\omega^I$ there is an extension $\sigma' \in T\upharpoonright_\sigma$  such that for every $n\in I$ we have $\sigma '(n)\neq \sigma(n)$.
\end{df}
Again, note that for every $T\in \mathbb{T_{\mathcal{H}}}$  we have that $[T]\notin\mathcal{H}$. We will show that such trees are canonical positive sets for the ideal $\mathcal{H}$ .
\begin{thm}
    For any analytic $A\subseteq\omega^{\omega}$ either $A\in \mathcal{H}$ or there is a tree $T \in \mathbb{T}_{\mathcal{H}}$ such that $[T]\subseteq A$. In particular, the assignment 
\begin{center}
        $\mathbb{T}_{\mathcal{H}}\ni T\longmapsto [T]\in Bor \setminus{\mathcal{H}}$
\end{center} is a dense embedding and the forcing notions $(\mathbb{T}_{\mathcal{H}}, \subseteq)$ and $(Bor /{\mathcal{H}}, \subseteq)$ are forcing equivalent.
\end{thm}
    \begin{proof}
The proof is identical to the proof of Theorem \ref{mmm} with the exception that for any tree $S$ on $(\omega \times \omega)^{<\omega}$, we define the derivative:
$$S'=\{(\sigma,\tau)\in S:\;\;\exists N^T_\sigma>|\sigma|\;\; \forall I\subseteq \omega\backslash N^T_\sigma\;\;\forall \rho \in \omega^I$$$$\exists( \sigma',\tau')\in S\upharpoonright_{(\sigma,\tau)}\;\;\forall n\in I \;\;\sigma'(n) \neq\rho(n)\}.$$

After iterating the derivative, if $\alpha_0<\omega_1$ is the least such that $T^{\alpha_0}=T^{\alpha_0+1}$, we consider the following two cases:

Case $1)$  $T^{\alpha_0}=\emptyset$.

Note that $[\pi (T)]\subseteq \bigcup_{\beta < \alpha_0} [\pi [T^\beta \backslash T^{\beta +1}]].$ It will be enough to show that for every $\beta<\alpha_0$, we have $[\pi[T^\beta\setminus T^{\beta+1}]]\in \mathcal{ H}.$ Fix $\beta<\alpha_0$ and $(\sigma, \tau)\in T^\beta \backslash T^{\beta+1}.$ We will show that $[\pi [T^\beta\upharpoonright_{(\sigma, \tau)}]]\in \mathcal{H}$.

Since $(\sigma, \tau )\notin T^{\beta+1}$  we get for every $M >|\sigma|$:

\begin{equation*} \exists I\subseteq \omega\backslash M \;\; \exists \rho \in \omega^I \;\; \forall (\sigma',\tau')\in T^\beta \upharpoonright_{(\sigma, \tau)} \;\;\;\;\;\;(*)\end{equation*}  \begin{equation*} \exists m\in I \;\;\sigma'(m)= \rho(m).\end{equation*}

 We will inductively build  an interval partition $J=(J_{n})_{n\in\omega}$ and $(x_n)_{n<\omega},$  such that $x_n\in \omega^{J_n}$ and $[\pi[T^\beta \upharpoonright_{(\sigma ,\tau )}]]\subseteq H_{J,x}.$ 

Assume that $(J_k)_{k\leq n+1}$ and $(x_k)_{k\leq n}$ where $x_k \in   \omega^{J_k}$ has been constructed. We use $(*)$ for $M =max(J_n) $. There is an interval $I\subseteq \omega \backslash \max(J_n)$ and $\rho \in \omega^{I}$ such that:
$$
\forall (\sigma',\tau') \in T^\beta \upharpoonright _{(\sigma,\tau)} \;\; \exists  m \in I\;\;\sigma'(m)=\rho(m) \;\;\;\;\;\;\;(**). 
$$

Let  $J_{n+1}=[\max(J_n)+1, \max (I)]$ and  let $x_{n+1}\upharpoonright _{ J_{n+1}}$ be such that $x_{n+1}\upharpoonright_ I =\rho$. Let $\bar{J}=\{J_n\}_{n<\omega}$ and $x=\bigcup_{n<\omega} x_n$.\\
We claim that $[\pi [T^\beta \upharpoonright _{(\sigma, \tau)}]] \subseteq H_{J,x}$. 

Assume that $y\in [\pi [T^\beta \upharpoonright _{(\sigma, \tau)}]]$ and let $z\in\omega^{\omega}$ be such that $(y,z)\in T^\beta \upharpoonright_{(\sigma,\tau)}$. By $(**)$, for every $n$ there exists $m\in I$ such that  $y(m)= \rho(m)$. Since $I\subseteq J_n$, for every $n\in\omega$, we can find $l\in I_n$ such that  $y(l)= x(l)$ and it holds for almost all $n\in\omega$. It follows that  $y\in H_{J,x}$. 

Case $2)$  $T^{\alpha_0}\neq \emptyset$. We will show that $\pi [T^{\alpha_0}] \in \mathbb{T}_{\mathcal{H}}$. Let $\sigma \in \pi [T^{\alpha_0}]$ and let $\tau$ be  such that $(\sigma, \tau) \in T^{\alpha_0}.$ As $T^{\alpha_{0}}=T^{\alpha_{0}+1}$ there is $N^T_\sigma>|\sigma|=|\tau|$ such that for every interval $I\subseteq \omega \backslash N^T_\sigma$ and for every $\rho\in\omega ^I$ we can find $(\sigma',\tau')\in T^{\alpha_0}\upharpoonright_{(\sigma,\tau)}$ and for every  $n\in I$ we have  $\sigma'(n)\neq\rho(n)$. 
    \end{proof}
    
We will next investigate the tree dichotomy related to the ideal $\mathcal{UM}$.
\begin{df}
    Let $\mathbb{T}_{\mathcal{UM}} $ be the collection of all trees $T\subseteq\omega^\omega$ such that for every node $\sigma\in T$ there is $N^T_\sigma>|\sigma|$ such that for every interval $I\subseteq \omega \backslash N^T_\sigma$, where $min(I)=N^T_\sigma$ and every $\psi: \omega^{N^T_\sigma}\rightarrow \omega^{I}$ there is an extension $\sigma' \in T\upharpoonright_\sigma$  such that $\psi(\sigma' \upharpoonright_{ min(I)})=\sigma'\upharpoonright _I.$
\end{df}
Note that for every $T\in \mathbb{T_{\mathcal{UM}}}$  we have that $[T]\notin\mathcal{UM}$. We will show that such trees are canonical positive sets for the $\mathcal{UM}$ ideal.

\begin{thm}
    For any analytic $A\subseteq\omega^{\omega}$ either $A\in \mathcal{UM}$ or there is a tree $T \in \mathbb{T}_{\mathcal{UM}}$ such that $[T]\subseteq A$. In particular, the assignment 
    \begin{center}
        $\mathbb{T}_{\mathcal{UM}}\ni T\longmapsto [T]\in Bor \setminus {\mathcal{UM}}$
    \end{center} is a dense embedding and the forcing notions $(\mathbb{T}_{\mathcal{UM}}, \subseteq)$ and $(Bor /{\mathcal{UM}}, \subseteq)$ are forcing equivalent.
    \end{thm}
    \begin{proof}
         The proof follows the same argument as the one for $\Mminus$, thus we only sketch the argument. Let $T$ be a tree on $(\omega \times \omega)^{<\omega}$ such that $A=\pi[[T]]$. For a tree $S$ on $(\omega \times \omega)^{<\omega}$ define the derivative:
  $$S'=\{(\sigma,\tau)\in S:
  \exists N^T_\sigma>|\sigma|\;\;\; \forall I\subseteq \omega\backslash N^T_\sigma, \;min(I)=N^T_\sigma\;\;\;\forall\psi: \omega^{min(I)}\rightarrow  \omega^{I}$$
  $$\exists (\sigma',\tau')\in S\upharpoonright_{(\sigma,\tau)}\;\; \psi(\sigma'\upharpoonright_{ N^T_\sigma })=\sigma'\upharpoonright _I\}.$$

After iterating the derivative, if $\alpha_0<\omega_1$ is the least such that $T^{\alpha_0}=T^{\alpha_0+1}$, we again consider the following two cases:

Case $1)$  $T^{\alpha_0}=\emptyset$.

Note that $[\pi (T)]\subseteq \bigcup_{\beta < \alpha_0} [\pi [T^\beta \backslash T^{\beta +1}]].$ It will be enough to show that for every $\beta<\alpha_0$ we have $[\pi[T^\beta\setminus T^{\beta+1}]]\in \mathcal{ UM}.$ Fix $\beta<\alpha_0$ and $(\sigma, \tau)\in T^\beta \backslash T^{\beta+1}.$ We will show that $[\pi [T^\beta\upharpoonright_{(\sigma, \tau)}]]\in \mathcal{UM}$.

Since $(\sigma, \tau )\notin T^{\beta+1}$ for every $M>|\sigma|$ we get:
$$\exists I\subseteq \omega\backslash N^T_\sigma \;\; min (I)=N^T_\sigma\;\;\exists \phi:\omega^{min(I)}\rightarrow  \omega^{I}\;\; $$
$$\forall (\sigma',\tau')\in T^\beta \upharpoonright_{(\sigma, \tau)} \;\; \sigma' \upharpoonright _I\neq \phi (\sigma' \upharpoonright_{ min(I)})\;\; \;\; (*)$$

 We will inductively build $\psi:\bigcup_n\omega^{min(J_n)}\rightarrow \bigcup_n \omega^{J_n}$ and an interval partition $\bar{J}=\{J_{n}\}_{n\in\omega}$ such that $[\pi[T^\beta \upharpoonright_{(\sigma ,\tau )}]]\subseteq M(\bar{J},\psi).$ 

Assume that $J_n$ and $\psi_n:\omega^{min(J_n)}\rightarrow \omega^{J_n}$ has been constructed. By $(*)$ there is an interval $I\subseteq \omega \backslash \max(J_n)$ and $\phi: \omega^{min(I)}\rightarrow \omega^{I} $ such that: 

$$
\forall (\sigma',\tau') \in T^\beta \upharpoonright _{(\sigma,\tau)} \;\; \sigma'\upharpoonright_ I \neq \phi(\sigma'\upharpoonright _{\min I})
\;\;\;(**). 
$$

Let  $J_{n+1}=I$ and  let $\psi_n$ be such that $\psi_n=\phi$. Let $\bar{J}=\{J_n\}_{n<\omega}$ and $\psi= \bigcup_n \psi_n$.
We claim that $[\pi [T^\beta \upharpoonright _{(\sigma, \tau)}]] \subseteq U(\psi, \bar{J})$. Take any $y\in [\pi [T^\beta \upharpoonright _{(\sigma, \tau)}]]$. There exists $z\in\omega^{\omega}$ such that $(y,z)\in T^\beta \upharpoonright_{(\sigma,\tau)}$. So by $(**)$  for every $n$ we have  $y\upharpoonright _I \neq \psi (y\upharpoonright _{J_n})$ It follows that  $y\in U(\psi, \bar{J})$.

Case $2)$  $T^{\alpha_0}\neq \emptyset$. We will show that $\pi [T^{\alpha_0}] \in \mathbb{T}_{\mathcal{UM}}$. Let $\sigma \in \pi [T^{\alpha_0}]$ and let $\tau$ be  such that $(\sigma, \tau) \in T^{\alpha_0}.$ As $T^{\alpha_{0}}=T^{\alpha_{0}+1}$ there is $N^T_\sigma>|\sigma|=|\tau|$ such that for every interval $I\subseteq \omega \backslash N^T_\sigma$, $\min(I)=N^T_\sigma$ and for every $\psi:\omega^{\min(I)}\rightarrow \omega^I$ we can find $(\sigma',\tau')\in T^{\alpha_0}\upharpoonright_{(\sigma,\tau)}$ and  we have  $\sigma'\upharpoonright _I=\psi(\sigma'\upharpoonright_ I)$. 
    \end{proof}

We will now show that, just like Miller trees, the tree types $\mathbb{T}_{\mathcal{M}_-}$ and $\mathbb{T}_{\mathcal{UM}}$ are far from being c.c.c.
\begin{proposition}
    $\mathbb{T}_{\mathcal{M}_-}$ and $\mathbb{T}_{\mathcal{UM}}$ have antichains of size $\continuum$ below each condition.
\end{proposition}
\begin{proof}
    We first handle $\mathbb{T}_{\Mminus}$. Fix $T\in\mathbb{T}_{\Mminus}$. We will need the following claim
    \begin{claim}
        For every $\sigma\in T$ there is a well-founded tree $F_{\sigma}\subseteq T\upharpoonright_{\sigma}$ such that for every interval $I\subseteq \omega\setminus(N^{T}_{\sigma}+1)$ and $\rho\in\omega^{I}$ exists $\tau\in F_{\sigma}$ such that $\tau\upharpoonright_{I}=\rho$.
    \end{claim}
    \begin{proof}[proof of the claim]
    Fix $\sigma\in T$ and let $\{(I_{n},\rho_{n}):n\in\omega\}$ be an enumeration of all pairs $(I,\rho)$ such that $I$ is an interval such that $\min I\geq N^{T}_{\sigma}$ and $\rho\in \omega^{I}$. We will inductively construct nodes $\{\tau_{n}:n\in\omega\}\subseteq T$ extending $\sigma$ such that for every $n\in\omega$ we have
    \begin{itemize}
        \item[--] $\tau_{n}$ does not extend any $\tau_{i}$, $i<n$,
        \item[--] $\tau_{n}\upharpoonright_{I_{n}}=\rho_{n}$.
    \end{itemize}
    Note that when $\tau_{n}$'s are constructed, then the tree $F_{\sigma}$ generated by $\{\tau_{n}:n\in\omega\}$ is as required.\\
    Assume that $\tau_{n}$ has been constructed and we want to construct $\tau_{n+1}$. Let $\rho_{n+1}'$ be such that  $dom(\rho_{n+1}')=I_{n+1}\cup\{\min I_{n+1}-1\}$, $\rho_{n+1}\subseteq\rho_{n+1}'$ and $\rho_{n+1}'(\min I_{n+1}-1)\notin\{\tau_{i}(\min I_{n+1}-1):i\leq n\}$. As $T\in\mathbb{T}_{\Mminus}$ there is $\tau_{n+1}\in T$ extending $\sigma$ such that $\tau_{n+1}(k)=\rho_{n+1}'(k)$ for every $k\in I_{n+1}\cup\{\min I_{n+1}-1\}$. Then $\tau_{n+1}$ is as required.
    \end{proof}
    Once the proof of the claim is complete, for every $x\in\Cantor$ we construct a tree $T_{x}\leq T$, $T_{x}\in\mathbb{T}_{\Mminus}$. For $\sigma\in T$ and $i\in2$, denote by $\sigma_{i}$ such $\tau\in T$ that $\tau\upharpoonright_{|\tau|-1}$ is a split node of $T$ and $\tau(|\tau|-1)$ is either the smallest or the second smallest element of $succ_{T}(\tau\upharpoonright_{|\tau|-1})$, depending on $i\in 2$. Fix $x\in\Cantor$ and inductively construct an increasing sequence of well-founded trees $\{H_{n}:n\in\omega\}$. Assuming that $H_{n}$ has been constructed let 
    \begin{center}
        $H_{n+1}=\bigcup \{F_{\sigma_{x(n+1)}}:\sigma$ is a terminal node of $H_{n}\}$.
    \end{center}
    When the induction is finished, define $T_{x}=\bigcup_{n}H_{n}$. By the claim we have that $T_{x}\in\mathbb{T}_{\Mminus}$. Also, it is not difficult to see that if $x\neq y$, then $[T_{x}]\cap [T_{y}]=\emptyset$.\\

    In the case of $\mathbb{T}_{\mathcal{UM}}$ the construction of the antichain is identical, but we use the following claim.
    \begin{claim}
        For every $\sigma\in T$ there is a well-founded tree $F_{\sigma}\subseteq T\upharpoonright_{\sigma}$ such that for every $m>N^{T}_{\sigma}$ and $\phi':\omega^{N^{T}_{\sigma}}\rightarrow\omega^{[N^{T}_{\sigma},m)}$ exists $\tau\in F_{\sigma}$ such that $\tau\upharpoonright_{[N^{T}_{\sigma},m)}=\phi'(\tau\upharpoonright_{N^{T}_{\sigma}})$.
    \end{claim}
    \begin{proof}[proof of the claim]
        Fix $\sigma\in T$ and let $\{(m_{n},\psi_{n}):n\in\omega\}$ be an enumeration of all pairs $(m,\psi)$ such that $m>N^{T}_{\sigma}$  and $\psi:\omega^{N^{T}_{\sigma}}\rightarrow\omega^{[N^{T}_{\sigma},m]}$. We will inductively construct nodes $\{\tau_{n}:n\in\omega\}\subseteq T$ extending $\sigma$ such that for every $n\in\omega$ we have
    \begin{itemize}
        \item[--] $\tau_{n}$ does not extend any $\tau_{i}$, $i<n$,
        \item[--] $\tau_{n}\upharpoonright_{[N^{T}_{n}, m_{n}]}=\psi_{n}(\rho_{n}\upharpoonright_{N^{T}_{\sigma}})$.
    \end{itemize}
    Note that when $\tau_{n}$'s are constructed, then the tree $F_{\sigma}$ generated by $\{\tau_{n}:n\in\omega\}$ is as required.\\
    Assume that $\tau_{n}$ has been constructed and we want to construct $\tau_{n+1}$. For any $\tau_{i}$ where $i\leq n$ pick $\tau_{n}'$ extending $\tau_{n}$ such that $\tau_{n}'\notin T$. Let also $m'>m$ be such that $m'>\max\{|\tau'_{i}|:i\leq n\}$. Define $\psi_{n+1}':\omega^{N^{T}_{\sigma}}\rightarrow\omega^{[N^{T}_{\sigma},m']}$ as follows: for $\tau\in \omega^{N^{T}_{\sigma}}$, if $\psi_{n+1}(\tau)$ extends $\tau_{i}$ for some $i\leq n$, then let $\psi'_{n+1}(\tau)\in \omega^{[N^{T}_{\sigma},m']}$ be such that $\tau^{\frown}\psi'_{n+1}(\tau)$ extends $\tau_{i}'$, and if $\psi_{n+1}(\tau)$ does not extend any $\tau_{i}$, $i\leq n$, then let $\psi'_{n+1}(\tau)$ be such that $\psi_{n+1}(\tau)\subseteq \psi'_{n+1}(\tau)$. By the definition of $T$ there is $\tau_{n+1}\in T$ extending $\sigma$ such that $\psi'_{n+1}(\tau_{n+1}\upharpoonright_{N^{T}_{\sigma}})=\tau_{n+1}\upharpoonright_{[N^{T}_{\sigma},m']}$. Note that as $\tau_{n+1}\in T$, it must be that $\tau_{n+1}$ does not extend any $\tau_{i}$ for $i\leq n$. Also $\psi_{n+1}(\tau_{n+1}\upharpoonright_{N^{T}_{\sigma}})=\tau_{n+1}\upharpoonright_{[N^{T}_{\sigma},m]}$ by the definition of $\psi'_{n+1}$. This finishes the induction step and the proof of the claim
    \end{proof}
    \end{proof}
The difference between the tree types for $\sigma$-ideals $\Mminus$, $\mathcal{UM}$ and the one for $\mathcal{H}$ is that $\mathbb{T}_{\mathcal{H}}$ may be a c.c.c. forcing below certain conditions (actually $\mathbb{T}_{\mathcal{H}}$ is equivalent to Cohen forcing below certain conditions) and below others it is equivalent to Miller forcing (see incoming work On star operation and some ideals in the Baire space part 2).\\

It follows from \cite{zapletal} that the forcing notions $\mathbb{T}_{\mathcal{M}_-}$, $\mathbb{T}_{\mathcal{H}}$, $\mathbb{T}_{\mathcal{UM}}$ are proper as all the three $\sigma$-ideals have $F_{\sigma}$ basis. We will show that these tree forcings allow fusion arguments and thus satisfies Axiom A.
\begin{df}
    A forcing notion $\mathbb{P}$ has Axiom A if there are orders $\{\leq_{n}:n\in\omega\}$ on $\mathbb{P}$ such that:
    \begin{itemize}
        \item[--] $\leq_{n+1}\subseteq\leq_{n}$ for every $n\in\omega$,
        \item[--] if $\{p_{n}:n\in\omega\}\subseteq\mathbb{P}$ is such that $p_{n+1}\leq_{n}p_{n}$ for every $n\in\omega$, then there is $q\in\mathbb{P}$ such that $q\leq_{n}p_{n}$ for each $n\in\omega$,
        \item[--] if $p\forces"\dot{x}\in\mathbb{V}"$ and $n\in\omega$, then there is $q\leq_{n}p$ and countable $X$ such that $q\forces"\dot{x}\in X"$.
    \end{itemize}
\end{df}
It is well-known that many standard tree forcing notions like Sacks, Miller and Laver have Axiom A and that Axiom A implies properness. We have the following fact.
\begin{thm}
    The forcing notions $\mathbb{T}_{\Mminus}$, $\mathbb{T}_{\mathcal{UM}}$ and $\mathbb{T}_{\mathcal{H}}$ have Axiom A.
\end{thm}
\begin{proof}
    As the fusion arguments for all three tree types are similar, we will only handle the first one in detail and sketch the other two.\\
    For $T,S\in\mathbb{T}_{\Mminus}$ and $n\in\omega$ define $T\leq_{n} S$ as:
    \begin{center}
        $T\leq S$, $T\cap \omega^{\leq n}=S\cap \omega^{\leq n}$ and for all $\sigma\in S\cap \omega^{\leq n}$ we have $N^{S}_{\sigma}=N^{T}_{\sigma}$.
    \end{center}
    Clearly $\leq_{n}$'s form decreasing sequence of orders on $\mathbb{T}_{\Mminus}$.\\
     To see the second condition of Axiom A assume that a sequence $\{T_{n}:n\in\omega\}\subseteq\mathbb{T}_{\Mminus}$ is such that $T_{n+1}\leq_{n} T_{n}$ for each $n\in\omega$. Define $S=\bigcap_{n}T_{n}$. We check that $S\in\mathbb{T}_{\Mminus}$. Assume that $\sigma\in S$. Let $n=|\sigma|+1$. We claim that $N^{S}_{\sigma}=N_{\sigma}^{T_{n}}$. To show that assume that let $I\subseteq\omega\setminus N^{T_{n}}_{\sigma}$ be an interval and $\rho\in\omega^{I}$. Notice that for $m=\max I$, as $N^{T^{n}_{\sigma}}=N_{\sigma}^{T_{m}}$, there is $\tau\in T_{m}\upharpoonright_{\sigma}$ such that $\tau\upharpoonright_{I}=\rho$ and then $\tau\in S$.\\
    To verify the third condition of Axiom A assume that $T\forces"\dot{x}\in\mathbb{V}"$ and $n\in\omega$. For every $\sigma\in T\cap \omega^{n}$ we will construct $S_{\sigma}\leq T\upharpoonright_{\sigma}$ and a countable set $X_{\sigma}$ such that 
    \begin{itemize}
        \item[--] $S_{\sigma}\forces"\dot{x}\in X_{\sigma}"$,
        \item[--] $N^{T}_{\sigma}=N^{S_{\sigma}}_{\sigma}$.
    \end{itemize}
    Notice that when such $S_{\sigma}$'s and $X_{\sigma}$'s are constructed, the tree $S=\bigcup_{\sigma\in T\cap \omega^{n}}S_{\sigma}$ and $X=\bigcup_{\sigma\in T\cap \omega^{n}}X_{\sigma}$ are such that $S\in\mathbb{T}_{\Mminus}$, $S\leq_n T$ and  $S\forces"\dot{x}\in X"$.\\
    Construction of $S_{\sigma}$ for $\sigma\in T\cap \omega^{n}$:\\
    Assume that $\sigma'$ is the leftmost node of the tree $T\upharpoonright_{\sigma}\cap\omega^{N^{T}_{\sigma}}$. For every $\tau\in T\upharpoonright_{\sigma}\cap \omega^{N^{T}_{\sigma'}}$ that does not extend $\sigma'$ find $R_{\tau}\leq T\upharpoonright_{\tau}$ that decides the value of $\dot{x}$. For a moment, fix any $\tau'\in T\upharpoonright_{\sigma}\cap \omega^{N^{T}_{\sigma'}}$ that does not extend $\sigma'$. For every $\tau\in T\upharpoonright_{\sigma}\cap\omega^{N^{R_{\tau'}}_{\tau}}$ find $R_{\tau}\leq T\upharpoonright_{\tau}$ that decides $\dot{x}$. Now, let $S_{\sigma}$ be union of all $R_{\tau}$ constructed above and let $X_{\sigma}$ collects all decisions for $\dot{x}$ constructed along the way. This concludes the construction of $S_{\sigma}$ and $X_{\sigma}$.\\
    In case of the forcing $\mathbb{T}_{\mathcal{UM}}$, $\mathbb{T}_{\mathcal{H}}$  the orders $\leq_{n}$ are defined in the same way, only the meaning of $N_\sigma^T$ changes.

    Similar arguments as in the case with $\mathbb{T}_{\Mminus}$ shows that the two orders defined above satisfies Axiom A.
\end{proof}
    
Recall that a real $x$ is called \emph{infinitely equal}  over a model $M$, if and only if  for every $ y\in \omega^\omega  \cap M$ there exists infinitely many $n$ such that $x(n)=y(n)$. Clearly, a Cohen real is an infinitely equal real but by result of Zapletal the opposite is not true. Also, adding iteratively two infinitely equal reals produces a Cohen real thus sometimes infinitely equal real are reffered to as half Cohen real. 
\begin{proposition}
    $\mathbb{T}_{\mathcal{M}_-}$ and $\mathbb{T}_{\mathcal{UM}}$ add infinitely equal real.
\end{proposition}
\begin{proof}
For any $f\in \omega^\omega$, by Proposition \ref{kkkkkkk}, we have that $E_f\subseteq\mathcal{M}_-$. So there is $S\leq T$ such that $[S]\cap E_f=\emptyset$. In particular $S\forces {"} \exists^{\infty}_{n}$ $  \dot{r}_{gen}(n)=f(n)"$    
\end{proof}

It is a well-known fact that if $\mathbb{P}$ adds an infinitely equal real and $\mathbb{P} \forces ^" \dot{\mathbb{Q}} $ adds infinitely equal real$^"$ then the two-step iteration $\mathbb{P}*\dot{\mathbb{Q}}$ adds a Cohen real. It follows that if $\mathbb{P}^{\omega_2} $ is a countable support iteration  of  length $\omega_2$ over a model of  $CH$ with $\mathbb{T}_{\mathcal{M}_-}$ as each iterand, then $\mathbb{V}^{\mathbb{P}_{\omega_2}} \forces (\cov(\mathcal{M})=\omega_2=\mathfrak{c})$. Also, in \cite{zapletal}, Zapletal showed that if a $\sigma$-ideal $\mathcal{I}$ has an $F_\sigma $-basis, which $\Mminus$, $\mathcal{UM}$, and $\mathcal{H}$ do have, then the idealized forcing notion $Bor/{\mathcal{I}}$ preserves Baire Category  (non-meager ground-model sets remain non-meager in the extension). On the other hand, in \cite{alek} Cieślak showed that in such a situation, preservation of the Baire property is carried out through countable support iteration; i.e., if we have a countable support iteration  $(\mathbb{P}_\alpha, \dot{\mathbb{Q}}_\alpha; \alpha<\omega_2)$ and for all $\alpha<\omega_2$ we have $\mathbb{P}_\alpha \forces (\dot{\mathbb{Q}}_\alpha =Bor/{\dot{\mathcal{I}}_\alpha} \text{    and } \dot{\mathcal{I}}_\alpha$ that has an  $F_\sigma$-basis $)$, then $\mathbb{P}_{\omega_2}$ preserves Baire Category. It follows that a countable support iteration of $\mathbb{T}_{\mathcal{M}_-}$ or $\mathcal{UM}$ forces the same constellation of Cichoń's diagram as a finite support iteration of Cohen forcing.
\begin{thm}
    Suppose that $\mathbb{T}$ is either $\mathbb{T}_{\Mminus}$ or $\mathbb{T}_{\mathcal{UM}}$. Then in the iterated $\mathbb{T}$-model (the model obtained by countable support iteration of $\mathbb{T}$ of length $\omega_2$ over a model of $CH$)
, we have $\non(\mathcal{M})=\omega_1<\omega_{2}=\cov(\Meager)=\continuum$.
\end{thm}

\begin{acknowledgements}
No artificial intelligence tools have been used during the work on this article.
\end{acknowledgements}

\printbibliography
\end{document}